\documentclass[12pt,twoside]{amsart}
\usepackage{latexsym,amsfonts,amsmath,amssymb,color}
\usepackage[titletoc]{appendix}
\numberwithin{equation}{section}

\def\Im{\mathop{\rm Im}\nolimits}
\def\Re{\mathop{\rm Re}\nolimits}

\def\ds{\displaystyle}
\def\ord{\mathop{\rm ord}\nolimits}

\def\R{\mathbb R}
\def\C{\mathbb C}
\def\N{\mathbb N}
\def\3|{\|\hspace*{-1pt}|}

\newcommand{\afrac}[2]{\genfrac{}{}{0pt}{1}{#1}{#2}}

\newcommand{\beqsn}{\arraycolsep1.5pt\begin{eqnarray*}}
\newcommand{\eeqsn}{\end{eqnarray*}\arraycolsep5pt}
\newcommand{\beqs}{\arraycolsep1.5pt\begin{eqnarray}}
\newcommand{\eeqs}{\end{eqnarray}\arraycolsep5pt}

\newtheorem{Th}{Theorem}[section]
\newtheorem{Rem}[Th]{Remark}
\newtheorem{Ex}[Th]{Example}
\newtheorem{Lemma}[Th]{Lemma}
\newtheorem{Def}[Th]{Definition}

\title{Semilinear p-evolution equations in Sobolev spaces}
\author[Ascanelli]{Alessia Ascanelli}
\address{Alessia Ascanelli\\
Dipartimento di Matematica e Informatica\\Universit\`a di Ferrara\\
Via Machiavelli n.~30\\
44121 Ferrara\\
Italy}
\email{alessia.ascanelli@unife.it}
\author[Boiti]{Chiara Boiti}
\address{Chiara Boiti\\
Dipartimento di Matematica e Informatica \\Universit\`a di Ferrara\\
Via Ma\-chia\-vel\-li n.~30\\
44121 Ferrara\\
Italy}
\email{chiara.boiti@unife.it}
\begin{document}


\begin{abstract}
We prove local in time well-posedness in Sobolev spaces of the Cauchy problem
for semi-linear $p$-evolution equations of the first order with real principal
part, but complex valued coefficients for the lower order
terms, assuming decay conditions on the imaginary parts as $|x|\to\infty$.
\end{abstract}

\maketitle
\markboth{\sc Semilinear p-evolution equations in Sobolev spaces}{\sc
A.~Ascanelli and C.~Boiti}

\section{Introduction and main result}
Given an integer $p\geq 2$, we consider the Cauchy problem
\beqs
\label{2}
\begin{cases}
P_u(D)u(t,x)=f(t,x),&(t,x)\in[0,T]\times\R\cr
u(0,x)=u_0(x),&x\in\R
\end{cases}
\eeqs
for the semi-linear $p$-evolution operator of the first order
\beqs
\label{1}
P_u(D)u=P(t,x,u(t,x),D_t,D_x)u:=&&D_tu +a_p(t)D_x^pu+\ds\sum_{j=0}^{p-1}a_j(t,x,u)
D_x^ju,
\eeqs
where $D=\frac 1i \partial$, $a_p\in C([0,T];\R)$, and $a_j\in C([0,T];
C^\infty(\R\times \C))$ with $x\mapsto a_j(t,x,w)\in{\mathcal B}^\infty(\R)$, for
$0\leq j\leq p-1$
(${\mathcal B}^\infty(\R)$ is the space of complex valued functions
which are bounded on $\R$ together with all their derivatives).

The condition that $a_p$ is real valued means that the principal symbol
(in the sense of Petrowski) of $P$ has the real characteristic
$\tau=-a_p(t)\xi^p$; by the Lax-Mizohata
theorem, this is a necessary condition to have a unique
solution, in Sobolev spaces, of the Cauchy problem \eqref{2} in a
neighborhood of $t=0$, for any $p\geq1$ (cf. \cite{M}).

The condition that $a_j(t,x,w)\in\C$ for $0\leq j\leq p-1$ imply some decay
conditions
on the coefficients as $|x|\to\infty$, because of the well-known necessary
condition of
Ichinose (cf. \cite{I1}) for $2-$evolution linear equations, and of the very
recent
necessary condition of \cite{ABZnec} for $p$-evolution linear equations with
arbitrary
$p\geq 2$.

We give the main result of this paper, Theorem \ref{th1}, for the
  operator $P$ in \eqref{1}, i.e. under the condition that $a_p$ does not
  depend on the monodimensional space variable $x\in\R$, only for simplicity's 
sake.
  Indeed, we can generalize Theorem \ref{th1} to the case
  $a_p=a_p(t,x)$, $x\in\R$: see Theorem~\ref{thxxx}.
Moreover, Theorems \ref{th1} and \ref{thxxx} can be generalized to the case
   $x\in\R^n$, up to some technical complications, following the ideas of
  \cite{KB}, \cite{CR2} and \cite{CC2}: see Remark \ref{Rn}.

 The monodimensional problem for either operator \eqref{1} or \eqref{Px}
  is of interest by itself; for instance, in the case $p=3$ our model
  recovers equations of Korteweg-de Vries type, widely used to describe
  the propagation of monodimensional waves of small amplitudes in
  waters of constant depth.

The original Korteweg-de Vries equation is
\beqs
\label{cdv}
\partial_tu=\frac32\sqrt{\frac gh}\partial_x\left(\frac12u^2+\frac23
\alpha u+\frac13\sigma\partial_x^2u\right),
\eeqs
where $u$ represents the wave elevation with respect to the water's surface,
$g$ is the gravity constant, $h$ the (constant) level of
water, $\alpha$ a fixed small constant and
\beqsn
\sigma=\frac{h^3}{3}-\frac{Th}{\rho g},
\eeqsn
with $T$ the surface tension, $\rho$ the density of the fluid.
The operator in \eqref{cdv} is of the form \eqref{1} with constant
coefficients $a_p=a_3$ and $a_2=0$, and with $a_1=a_1(u)$.

Assuming that the water's level $h$ depends on $x$, we are lead to
an operator of the form \eqref{Px}. Our model can so be applied to study
the evolution of the wave when the seabed is variable.

The $n-$dimensional problem can find applications in the study
of Schr{\"o}dinger type equations ($p=2$), and of higher order equations  that
can be factorized into the product of Schr{\"o}dinger type equations, as
for instance the Euler-Bernoulli equation of a vibrating plate 
\[\partial_t^2u+a^2(t)\Delta_x^2u+\ds\sum_{|\alpha|\leq 3}b_\alpha(t,x)
\partial_x^\alpha u=0\]
(vibrating beam in the case $n=1$).
The above operator factorizes into the product 
$(\partial_t-ia(t)
\Delta_x-b(t,x,\partial_x))(\partial_t+ia(t)\Delta_x+b(t,x,\partial_x))$ of
two (pseudo-differential) Schr{\"o}dinger type operators, modulo terms of
order zero.

The aim of this paper is to give sufficient decay conditions on the
coefficients of $P_u(D)$ in order that the Cauchy problem \eqref{2} is
locally in time
well-posed in $H^\infty$.

Results of $H^\infty$ well-posedness of the Cauchy problem for linear
$p$-evolution equations of the first order
\beqs
\label{1lin}
D_tu +a_p(t)D_x^pu+\ds\sum_{j=0}^{p-1}a_j(t,x)D_x^ju=f(t,x),
\eeqs
or for linear $p$-evolution equations of higher order,
have already been obtained, first for real valued (or complex valued,
with imaginary part not depending on $x$) coefficients (see for
instance \cite{A2, AC} and the references therein), then for complex
valued coefficients depending on the
space variable $x$
under suitable decay conditions on the coefficients as $|x|\to+\infty$
(see \cite{I2,KB,CC1,ABZ,AB1} for equations of the form \eqref{1lin},
\cite{AB2,ACC,CR1,CR2,T1} for higher order equations, and \cite{ACa}
for equation \eqref{1lin} in a different framework). Among all these
results, in the present paper
we shall need an extension of the following theorem of
\cite{ABZ} (see Theorem~\ref{iop} below):
\begin{Th}
\label{thabz}
Let us consider the operator (\ref{1lin}) with $a_p\in C([0,T];\R)$, $a_p(t)\geq0\ \forall t\in[0,T]$ and
$a_j\in C([0,T];{\mathcal B}^\infty)$ for $0\leq j\leq p-1$. Suppose that there exists $C>0$ such that $\forall(t,x)\in[0,T]\times\R$:
\beqs
\label{3}
\quad&&|\Re D_x^\beta a_j(t,x)|\leq C a_p(t),\qquad 0\leq \beta\leq j-1,\ 3\leq j\leq p-1,\\
\label{4}
&&|\Im D_x^\beta a_j(t,x)|\leq\frac{Ca_p(t)}{\langle
x\rangle^{\frac{j-[\beta/2]}{p-1}}},\qquad 0\leq\left[\frac\beta2\right]\leq j-1,\ 3\leq j\leq p-1,\\
\label{5}
&&|\Im a_2(t,x)|\leq \frac{Ca_p(t)}{\langle x\rangle^{\frac{2}{p-1}}},
\\
\label{6'}
&&|\Im a_1(t,x)|+|\Im D_x a_2(t,x)|\leq
\frac{Ca_p(t)}{\langle x\rangle^{\frac{1}{p-1}}},
\eeqs
where $[\beta/2]$ denotes the integer part of $\beta/2$ and
$\langle x\rangle:=\sqrt{1+x^2}$.

Then, the Cauchy problem associated to equation \eqref{1lin} with
$u(0,x)=u_0(x)$,
$x\in\R$,
is well-posed in $H^\infty$ (with loss of derivatives). More precisely,
there exists a positive constant $\sigma$ such that for all
$f\in C([0,T];H^s(\R))$ and $u_0\in H^s(\R)$ there is a unique solution
$u\in C([0,T];H^{s-\sigma}(\R))$ which satisfies the following energy
estimate:
\beqs
\label{Eold}
\|u(t,\cdot)\|^2_{s-\sigma}\leq C_s\left(\|u_0\|^2_s+
\int_0^t\|f(\tau,\cdot)\|^2_s \,d\tau\right)\qquad
\forall t\in[0,T],
\eeqs
for some $C_s>0$, with $\|\cdot\|_s=\|\cdot\|_{H^s}$.
\end{Th}

\begin{Rem}
\begin{em}
Condition \eqref{4} with $j=p-1,\ \beta=0$, i.e.
$$|\Im a_{p-1}(t,x)|\leq Ca_p(t)\langle x\rangle^{-1},$$
is strictly consistent with the necessary condition
\beqsn
\exists M,N>0:\quad \sup_{x\in\R}\min_{0\leq\tau\leq t\leq T}\int_{-\varrho}^\varrho
\Im a_{p-1}(t,x+p a_p(\tau)\theta)d\theta\leq M\log(1+\varrho)+N,\quad
\forall \varrho>0
\eeqsn
for $H^\infty$ well-posedness of the Cauchy problem associated to
equation \eqref{1lin}, proved in \cite{ABZnec}.
\end{em}
\end{Rem}

As far as we know, semi-linear equations $P_u(D)u=f$ of the form
\eqref{1}, or of higher order, have been considered
in the case of complex valued coefficients with imaginary part not
depending on $x$
(see, for instance, \cite{AZ}), or in the hyperbolic case (see, for
instance, \cite{A,CNT}). Recently, we considered in \cite{ABZpad} semi-linear 3-evolution
equations of the first order and in \cite{AB3} a 2-evolution equation of
order 2 that generalizes the Boussinesq equation.
We gave sufficient decay conditions on the coefficients for $H^\infty$
well-posedness of the Cauchy problem (see the comments after Theorem \ref{th1}).

Here we consider the general case ($p\geq2$) of non-linear $p$-evolution
equations of the first order, proving the following:

\begin{Th}
\label{th1}
Let us assume that there exist constants $C_p>0$ and $C>0$ and a function
$\gamma:\ \C\to\R^+$ of class $C^7$
such that the coefficients of the semi-linear equation \eqref{1} satisfy
for all $(t,x,w)\in[0,T]\times\R\times\C$:
\beqs
\label{ap}
&&\quad a_p(t)\geq C_p,
\\
\label{im}
&&\quad \ |\Im (D_x^\beta a_j)(t,x,w)|\leq
\frac{C \gamma(w)}{\langle x\rangle^{\frac{j-[\beta/2]}{p-1}}},\quad 0\leq
\left[\frac\beta 2\right]\leq j-1,\ 3\leq j\leq p-1,
\\
\label{re}
&&\quad \ |\Re (D_x^\beta a_j)(t,x,w)|\leq C \gamma(w)\qquad 0\leq \beta\leq j-1,\
3\leq j\leq p-1,
\\
\label{dew}
&&\quad \ | (D_w^\gamma D_x^\beta a_j)(t,x,w)|\leq
\frac{C \gamma(w)}{\langle x\rangle^{\frac{j-[(\gamma+\beta)/2]}{p-1}}},\quad
\gamma\geq 1, \beta\geq 0, \left[\frac{\gamma+\beta} 2\right]\leq j-1,\
3\leq j\leq p-1,
\\
\label{Rea2}
&&\quad \ |\Re a_2(t,x,w)|\leq C\gamma(w),
\\
\label{a2}
&&\quad \ |\Im a_2(t,x,w)|\leq \frac{C\gamma(w)}{\langle x\rangle^{\frac{2}{p-1}}},\\
\label{a1}
&&\quad \ |\Im a_1(t,x,w)|+|\Im D_x a_2(t,x,w)|+|D_w a_2(t,x,w)|\leq
\frac{C\gamma(w)}{\langle x\rangle^{\frac{1}{p-1}}}.
\eeqs
Then the Cauchy problem \eqref{2} is locally in time well-posed in $H^\infty$:
for all $f\in C([0,T];H^\infty(\R))$ and $u_0\in H^\infty(\R)$, there
exists $0<T^*\leq T$ and a unique solution $u\in C([0,T^*];
H^\infty(\R))$ of \eqref{2}.
\end{Th}

Notice that conditions \eqref{ap}-\eqref{a1} correspond exactly
    to \eqref{3}-\eqref{6'} for linear equations. In
    \cite{ABZpad}
we had to strengthen the leading condition \eqref{im} for $\beta=0$,
    $j=p-1$, $p=3$, requiring a decay of order $1+\varepsilon$
instead of order 1 on the sub-leading coefficient:
    $$|\Im a_2(t,x,w)|\leq\frac{C\gamma(w)}{\langle x\rangle^{1+\varepsilon}},
\quad
    \epsilon>0.$$
Also in \cite{AB3} such a faster decay is required.
This is because, both in \cite{ABZpad} and \cite{AB3}, we used a fixed point
argument to show the existence of a solution to the semi-linear equation.
To avoid such a stronger condition, here we make use of the Nash-Moser theorem. Inspired by \cite{DA}
(see also \cite{CNT}), the idea is
to linearize the equation, fixing $u\in C([0,T], H^\infty(\R))$ and solving
the linear Cauchy problem in the unknown $v(t,x)$
\beqs
\label{2lin}
\begin{cases}
P_u(D)v(t,x)=f(t,x),&(t,x)\in[0,T]\times\R\cr
v(0,x)=u_0(x),&x\in\R,
\end{cases}
\eeqs
and then apply the Nash-Moser theorem.

The paper is organized as follows. In Section \ref{sec3} we
briefly retrace the proof of Theorem \ref{thabz} for the linear problem \eqref{2lin} taking care of the
dependence on $u$ of the constants in the energy estimate \eqref{Eold}.
Then, in Section \ref{sec4}, we prove Theorem \ref{th1}.
Section \ref{gen} is devoted to generalizations of the main result
to the cases $a_p=a_p(t,x)$ and/or $x\in\R^n$.
Finally, in Appendix \ref{app:a}
we collect the main notions about tame spaces and the Nash-Moser
Theorem~\ref{thNM}, according to \cite{DA}, \cite{H}; in Appendix \ref{app:b} we recall the main tools used to write an energy estimate for the solution $v$ of the linear Cauchy problem \eqref{2lin}, i.e. Sharp-G{\aa}rding theorem (with remainder in explicit form, see \cite[Ch.~3, Thm.~4.2]{KG}) and
Fefferman-Phong ine\-qua\-li\-ty, \cite{FP}.


\section{The linearized problem}
\label{sec3}

In this section we consider, for a fixed function $u\in C([0,T]; H^\infty(\R))$,
the linear Cauchy problem \eqref{2lin} in the unknown $v$.
By Theorem \ref{thabz}, we know that this Cauchy problem is well-posed with
loss of derivatives in Sobolev spaces, and the solution $v$ satisfies the
energy estimate \eqref{Eold}, with a positive constant $C_s=C_s(u)$.
The aim of this section is to retrace as briefly as possible the proof of
Theorem \ref{thabz} taking care of the dependence of the operator $P_u(D)$ on $u$,
to compute precisely $C_s(u)$: this will be needed in the proof of
Theorem~\ref{th1}. The computation of $C_s(u)$ clarifies that to write an energy estimate of the form \eqref{Eold} it is enough to fix $u\in C([0,T];H^{4p-3}(\R))$.

The present section is devoted to prove the following:
\begin{Th}\label{iop}
  Under the assumptions of Theorem \ref{th1}, there exists $\sigma>0$ such
  that for every $u\in C([0,T];H^{4p-3}(\R))$,
  $f\in C([0,T]; H^s(\R))$ and $u_0\in H^s(\R)$, there exists a unique solution
  $v\in C([0,T]; H^{s-\sigma}(\R))$ of the Cauchy problem \eqref{2lin}
  and the following energy
  estimate is satisfied:
\beqs
\label{Eoldv}
\|v(t,\cdot)\|^2_{s-\sigma}\leq C_{s,\gamma}e^{(1+\|u\|^{4p-3}_{4p-3})t}
\left(\|u_0\|^2_s+
\int_0^t\|f(\tau,\cdot)\|^2_s \,d\tau\right)\quad
\forall t\in[0,T],
\eeqs
for some $C_{s,\gamma}>0$.
\end{Th}

The proof of Theorem \ref{iop} is based on the energy method, after an
appropriate change of variable of the form
\beqs\label{cvar}
v(t,x)=e^{\Lambda(x,D)}w(t,x),
\eeqs
and makes use of the sharp-G{\aa}rding Theorem and the
Fefferman-Phong inequality, see Appendix~\ref{app:b}.

The operator $\Lambda$ is suitably constructed as a real valued,
invertible on $L^2$ operator with symbol $e^{\Lambda(x,\xi)}\in S^\delta$,
$\delta>0$. The change of variable \eqref{cvar} transforms the Cauchy
problem (\ref{2lin}) into the equivalent problem
\beqs
\label{6}
\begin{cases}
P_\Lambda(t,x,u(t,x), D_t,D_x) w(t,x)=f_\Lambda(t,x) \quad (t,x)\in [0,T]\times\R\\
w(0,x)=u_{0,\Lambda}(x) \quad x\in\R
\end{cases}
\eeqs
for $$P_\Lambda:=(e^\Lambda)^{-1}Pe^\Lambda,\quad
f_\Lambda:=(e^\Lambda)^{-1}f,\quad u_{0,\Lambda}:=(e^\Lambda)^{-1}u_0.$$
We are so reduced to show the well-posedness of (\ref{6}) in Sobolev spaces,
which is equivalent, since $e^\Lambda$ has order $\delta>0$, to the desired
well-posedness with loss of $\sigma=2\delta$ derivatives for the Cauchy
problem \eqref{2lin}.

The operator $\Lambda(x,D_x)$ is a pseudo-differential operator having symbol
\beqs
 \label{240}
\Lambda(x,\xi):=\lambda_{p-1}(x,\xi)+\lambda_{p-2}(x,\xi)+
\ldots+\lambda_1(x,\xi),
\eeqs
with
\beqs
\label{26}
\lambda_{p-k}(x,\xi):=M_{p-k}\omega\left(\frac\xi h\right)\int_0^x
\langle y\rangle^{-\frac{p-k}{p-1}}\psi\left(
\frac{\langle y\rangle}{\langle\xi\rangle_h^{p-1}}\right)dy
\langle\xi\rangle_h^{-k+1}\,,
\quad 1\leq k\leq p-1,
\eeqs
where $\langle \xi\rangle_h:=\sqrt{h^2+\xi^2}$ for $h\geq1$,
$\langle y\rangle:=\langle y\rangle_1$,
$M_{p-k}>0$ will be chosen large enough
throughout the proof, $\omega$ is a $C^\infty(\R)$ function such that
$\omega(y)=0$ for $|y|\leq 1$ and $\omega (y)=|y|^{p-1}/y^{p-1}$ for
$|y|\geq 2$, $\psi\in C^\infty_0(\R)$ is such that $0\leq\psi(y)\leq 1$ for
all $y\in\R$,
$\psi(y)=1$ for $|y|\leq\frac 12$, and $\psi(y)=0$ for $|y|\geq1$.

In the following Lemma we list the main properties of the symbols
$\lambda_k(x,\xi)$ and $\Lambda(x,\xi)$, referring to \cite{ABZ} for all
proofs.
\begin{Lemma}
  \label{lemma1} Let $\chi_{E}$ be the characteristic function of the
  set $E=\{\langle y\rangle \leq \langle\xi\rangle_h^{p-1}\}\subset\R^2$. We have that
\beqsn
|\lambda_{p-1}(x,\xi)|&\leq& M_{p-1}\log2+M_{p-1}(p-1)\log\langle\xi\rangle_h,
\\
|\lambda_{p-k}(x,\xi)|&\leq& M_{p-k}\frac{p-1}{k-1}
\langle x\rangle^{\frac{k-1}{p-1}}\langle \xi\rangle_h^{-k+1}\chi_{E}(x),
\quad 2\leq k\leq p-1;
\eeqsn
moreover, for every $\alpha\neq 0,\ \beta\neq 0$, there exist
$C_{\alpha,\beta}>0$ such that:
\beqs
\nonumber
|\partial_x^\beta\lambda_{p-k}(x,\xi)|&\leq& C_{0,\beta}
\langle x\rangle^{\frac{k-1}{p-1}-\beta}\langle\xi\rangle_h^{-k+1}\chi_{E}(x),
\quad 1\leq k\leq p-1,
\\
\nonumber
|\partial_\xi^\alpha\lambda_{p-1}(x,\xi)|&\leq& C_{\alpha,0}M_{p-1}
\langle\xi\rangle_h^{-\alpha}(1+\log\langle\xi\rangle_h\chi_{\{|\xi|<2h\}}),
\\
\label{PK2}
|\partial_\xi^\alpha\lambda_{p-k}(x,\xi)|&\leq& C_{\alpha,0} M_{p-k}
\langle x\rangle^{\frac{k-1}{p-1}}
\langle\xi\rangle_h^{-\alpha-k+1}\chi_{E}(x),\quad 2\leq k\leq p-1,
\\
\nonumber
|\partial_\xi^\alpha\partial_x^\beta\lambda_{p-k}(x,\xi)|&\leq&
C_{\alpha,\beta}\langle x\rangle^{\frac{k-1}{p-1}-\beta}
\langle\xi\rangle_h^{-\alpha-k+1}\chi_{E}(x),\quad1\leq k\leq p-1.
\eeqs
Again, there exist positive constants $C,\delta$ independent on $h$ and,
for every $\alpha\in\N_0:=\N\cup\{0\}$ and $\beta\in\N$, positive constants
$\delta_{\alpha,\beta}$, independent on $h$, such that:
\beqs
\label{7}
|\Lambda(x,\xi)|&\leq& C+\delta\log\langle\xi\rangle_h,
\\
\label{8}
|\partial_\xi^\alpha D_x^\beta\Lambda(x,\xi)|&\leq&
\delta_{\alpha,\beta}\langle\xi\rangle_h^{-\alpha}.
\qquad
\eeqs
\end{Lemma}
By (\ref{7}) and (\ref{8}), it can be easily shown that the operator
$e^{\Lambda(x,D_x)}$ with symbol
$e^{\Lambda(x,\xi)}\in S^\delta$ is invertible for $h$ large enough, say
$h\geq h_0$, with inverse
\beqs
\label{sole}
(e^\Lambda)^{-1}=e^{-\Lambda}(I+R)
\eeqs
where $I$ is the identity operator and $R$ is the operator defined by means
of the convergent Neumann series $R=\sum_{n=1}^{+\infty}r^n$, with principal 
symbol
\beqsn
r_{-1}(x,\xi)=\partial_\xi\Lambda(x,\xi)D_x\Lambda(x,\xi)\in S^{-1}.
\eeqsn
Moreover, for $h\geq h_0$ we have
\beqs
\label{27}
&&|\partial_\xi^\alpha e^{\pm\Lambda(x,\xi)}|\leq C_\alpha
\langle\xi\rangle_h^{-\alpha}e^{\pm\Lambda(x,\xi)}
\qquad\forall\alpha\in\N_0,\\
\label{28}
&&|D_x^\beta e^{\pm\Lambda(x,\xi)}|\leq C_\beta
\langle x\rangle^{-\beta}e^{\pm\Lambda(x,\xi)}
\qquad\forall\beta\in\N_0.
\eeqs
Finally, denoting $iP=\partial_t+A(t,x, u(t,x), D_x),$ with
$$A(t, x,u(t,x), \xi)=ia_p(t)\xi^p+\sum_{j=0}^{p-1}ia_j(t,x,u(t,x))\xi^j,$$
we have that
 \beqs\label{ipilambda}
 iP_\Lambda=\partial_t+(e^{\Lambda(x,D_x)})^{-1}A(t,x,u, D_x)e^{\Lambda(x,D_x)}
 \eeqs
 with
\beqs\label{li}
(e^{\Lambda(x,D_x)})^{-1}A(t,x,u, D_x)e^{\Lambda(x,D_x)}&=&e^{-\Lambda(x,D_x)}A(t,x,u, D_x)
e^{\Lambda(x,D_x)}\\
\nonumber
&+&\sum_{m=0}^{p-2}\frac{1}{m!}\sum_{n=1}^{p-1-m}
e^{-\Lambda(x,D_x)}A^{n,m}(t,x,u, D_x)e^{\Lambda(x,D_x)}
\\
\nonumber
&+&A_0(t,x,D_x),
\eeqs
where $A_0(t,x,\xi)\in S^0$ and $A^{n,m}(t,x,u,\xi)=\partial_\xi^mr^n(x,\xi)
D_x^m(A(t,x,u,\xi))
\in S^{p-m-n}$ (see \cite[Lemma~2.6]{ABZ}).
The lower order terms in \eqref{li} have the same structure as the principal
term,
so the structure of the operator $(e^{\Lambda})^{-1}Ae^{\Lambda}$ is the same as
that of $e^{-\Lambda}Ae^{\Lambda}$.

In the proof of Theorem \ref{iop}, the following Lemma will be crucial:
\begin{Lemma}\label{lemde}
  Under the assumptions \eqref{im}, \eqref{re}, \eqref{dew} there exists a positive constant $C'$ such that for every fixed $u\in C([0,T]; H^{4p-3}(\R))$ the coefficients $a_j(t,x,u(t,x))$ of the operator $P_u(D)$ satisfy $\forall (t,x)\in [0,T]\times\R$:
\beqs
\label{euna}&&\qquad\ |\Re D_x^\beta(a_j(t,x,u(t,x)))|\leq
C'\gamma(u)(1+\|u\|_{1+\beta}^\beta),
\quad 1\leq \beta\leq j-1,\,3\leq j\leq p-1
\\
\label{edue}&&\qquad\ |\Im D_x^\beta(a_j(t,x,u(t,x)))|\leq
\frac{C'\gamma(u)}{\langle x\rangle^{\frac{j-[\beta/2]}{p-1}}}(1+\|u\|_{1+\beta}^\beta),
\quad \beta\geq 1,\, \left[\frac\beta2\right]\leq j-1,\, 3\leq j\leq p-1.
\eeqs
\end{Lemma}
\begin{proof}
Let us compute for $3\leq j\leq p-1$ and $\beta\geq 1$ the derivative
\beqsn
D_x^\beta (a_j(t,x,u))=(D_x^\beta a_j)(t,x,u)+\ds
\sum_{\afrac{\beta_1+\beta_2=\beta}{\beta_2\geq 1}}c_\beta
\sum_{\afrac{r_1+\ldots+r_{q}=\beta_2}{r_i\geq1}}c_{q,r}
\partial^q_w (D_x^{\beta_1}a_j)(t,x,u)(D_x^{r_1}u)\cdots (D_x^{r_q}u)
\eeqsn
for some $c_\beta,c_{q,r}>0$.

From conditions \eqref{re} and \eqref{dew}, using the relationship between
geometric and arithmetic mean value and Sobolev's inequality
(see \cite[Ch.~3, Lemma~2.5]{KG})
for the fixed function $u(t)\in H^{4p-3}(\R)$, we immediately get that
for every  $3\leq j\leq p-1$, $\beta\geq 1$ and $n\geq 1$
\beqs
\nonumber
|\Re D_x^\beta (a_j(t,x,u))|&\leq& |\Re(D_x^\beta a_j)(t,x,u)|
\\
\nonumber
&&+\ds\sum_{\afrac{\beta_1+\beta_2=\beta}{\beta_2\geq 1}}c_\beta
\sum_{\afrac{r_1+\ldots+r_{q}=\beta_2}{r_i\geq1}}c_{q,r}|
(\partial^q_w D_x^{\beta_1}a_j)(t,x,u)|\cdot |D_x^{r_1}u|\cdots |D_x^{r_q}u|
\\
\nonumber
&\leq& C\gamma(u)\left(1+\ds\sum_{\afrac{\beta_1+\beta_2=\beta}{\beta_2\geq 1}}c_\beta
\sum_{\afrac{r_1+\ldots+r_{q}=\beta_2}{r_i\geq1}}c_{q,r}|D_x^{r_1}u|\cdots |D_x^{r_q}u|\right)
\\
\nonumber
&\leq&C\gamma(u)\left(1+\ds\sum_{\afrac{\beta_1+\beta_2=\beta}{\beta_2\geq 1}}c_\beta
\sum_{\afrac{r_1+\ldots+r_{q}=\beta_2}{r_i\geq1}}c_{q,r}\left( \frac{|D_x^{r_1}u|+\cdots
  +|D_x^{r_q}u|}{q}\right)^q\right)
\\
\nonumber
&\leq&C\gamma(u)\left(1+\ds\sum_{\afrac{\beta_1+\beta_2=\beta}{\beta_2\geq 1}}c_\beta
\sum_{\afrac{r_1+\ldots+r_{q}=\beta_2}{r_i\geq1}}c_{q,r}\|u\|^q_{n+\beta_2}\right)
\\
\label{VAR2}
&\leq & C'\gamma(u)(1+\|u\|^\beta_{n+\beta});
\eeqs
taking $n=1$ we get \eqref{euna}. On the other hand, from \eqref{im} and
\eqref{dew} we have, with similar computations, that for every $\beta\geq 1$
such that $[\beta/2]\leq j-1$, $3\leq j\leq p-1$ and $n\geq 1$
\beqs
\nonumber
|\Im(D_x^\beta (a_j(t,x,u))|&\leq& |\Im((D_x^\beta a_j)(t,x,u))|
\\
\nonumber
&&+\ds\sum_{\afrac{\beta_1+\beta_2=\beta}{\beta_2\geq 1}}c_\beta
\sum_{\afrac{r_1+\ldots+r_{q}=\beta_2}{r_i\geq1}}c_{q,r}|(\partial^q_w D_x^{\beta_1}a_j)
(t,x,u)|\cdot |D_x^{r_1}u|\cdots |D_x^{r_q}u|
\\
\nonumber
&\leq & \frac{C\gamma(u)}{\langle x \rangle^{\frac{j-[\beta/2]}{p-1}}}
\left(1+ \ds\sum_{\afrac{\beta_1+\beta_2=\beta}{\beta_2\geq 1}}c_\beta
\sum_{\afrac{r_1+\ldots+r_{q}=\beta_2}{r_i\geq1}}c_{q,r}|D_x^{r_1}u|\cdots |D_x^{r_q}u|\right)
\\
\label{VAR3}
&\leq & \frac{C'\gamma(u)}{\langle x \rangle^{\frac{j-[\beta/2]}{p-1}}}
(1+\|u\|^{\beta}_{n+\beta});
\eeqs
taking $n=1$ we have \eqref{edue}.
\end{proof}

\begin{proof}[Proof of Theorem \ref{iop}]
We divide the proof of Theorem \ref{iop} into the following steps:
\begin{itemize}
\item[{\bf Step 1.}]
Compute the symbol of the operator $e^{-\Lambda}Ae^\Lambda$ and show that its
terms of order $p-k$, $1\leq k\leq p-1$, denoted by
$\left.(e^{-\Lambda}Ae^\Lambda)\right|_{\ord(p-k)}$,
satisfy
\beqs
\label{stp3}
\left|\left.\Re (e^{-\Lambda}Ae^\Lambda)\right|_{\ord(p-k)}(t,x,u,\xi)\right|\leq
C_{(M_{p-1},\ldots,M_{p-k})}(u)\langle x\rangle^{-\frac{p-k}{p-1}}
\langle\xi\rangle_h^{p-k}
\eeqs
for a positive constant
\beqsn
C_{(M_{p-1},\ldots,M_{p-k})}(u)=C_{M_{p-1},\ldots,M_{p-k}}(1+\gamma(u))
(1+\|u\|_{1+(k-1)}^{k-1})
\eeqsn
depending on $u$, $M_{p-1},\ldots,M_{p-k}$ and not on $M_{p-k-1},\ldots,M_1$.
This allows the constants $M_{p-1},\ldots,M_1$ to be chosen recursively,
since at each step (say ``step $p-k$'') we have something which depends only
on the already chosen $M_{p-1},\ldots,M_{p-k+1}$
and on the new $M_{p-k}$ that is going to be chosen, and not on the constants
$M_{p-k-1},\ldots,M_1$ which will be chosen in the next steps.

\item[{\bf Step 2.}]
We choose $M_{p-1}>0$ such that
\beqsn
\Re\left.(e^{-\Lambda}Ae^\Lambda)\right|_{\ord(p-1)}+\tilde{C}(u)\geq0
\eeqsn
for some positive constant $\tilde{C}(u)>0$ depending on $u$, and apply the
sharp-G{\aa}rding Theorem~\ref{thA1}
to $\left.(e^{-\Lambda}Ae^\Lambda)\right|_{\ord(p-1)}+\tilde{C}(u)$ to get
\beqsn
\sigma(e^{-\Lambda}Ae^\Lambda)=ia_p\xi^p+Q_{p-1}
+\sum_{k=2}^{p-1}\left.(e^{-\Lambda}Ae^\Lambda)\right|_{\ord(p-k)}
+R_{p-1}+A_0,
\eeqsn
where  $Q_{p-1}$ is a positive operator of order $p-1$,
$A_0\in S^0$ and $R_{p-1}(t,x,u,D_x)$ is a remainder of order $p-2$,
consisting in the sum of terms satisfying \eqref{stp3} with a new constant
which is similar to $C_{(M_{p-1},\ldots,M_{p-k})}(u)$ but depends on a higher number
of derivatives of the fixed function $u$.
\item[{\bf Step 3.}]
  To iterate this process, finding positive constants $M_{p-2},\ldots,M_1$ such
  that
\beqsn
\Re\left.(e^{-\Lambda}Ae^\Lambda)\right|_{\ord(p-k)}+\tilde{C}(u)\geq0
\eeqsn
for some $\tilde{C}(u)>0$ depending on $u$, we need
to investigate the action of the sharp-G{\aa}rding Theorem~\ref{thA1}
to each term
of the form
\beqsn
\left.(e^{-\Lambda}Ae^\Lambda)\right|_{\ord(p-k)}+S_{p-k},
\eeqsn
where $S_{p-k}$ denotes terms of order $p-k$ coming from
remainders of
previous applications of
the sharp-G{\aa}rding Theorem, for $p-k\geq3$.

We show at this step that remainders are sums of terms
with ``the right decay at the right level'', in the sense that they satisfy
(\ref{stp3}) with a new constant which is similar to $C_{(M_{p-1},\ldots,M_{p-k})}(u)$
but depends on a higher number of derivatives of the fixed function $u$.
Then we apply the sharp-G{\aa}rding Theorem to terms of order $p-k$,
up to order $p-k=3$.

\item[{\bf Step 4.}]
  In this step we apply the Fefferman-Phong inequality~\eqref{A3}
  to terms of order
$p-k=2$ and the sharp-G{\aa}rding inequality~\eqref{A2} to terms of order
$p-k=1$, finally obtaining that
\beqsn
\sigma(e^{-\Lambda}Ae^\Lambda)=ia_p\xi^p
+\sum_{s=1}^pQ_{p-s}
\eeqsn
with
\beqsn
&&\Re\langle Q_{p-s}w,w\rangle\geq0\qquad\forall w(t,\cdot)\in H^{p-s},\quad
s=1,\ldots,p-3\\
&&\Re\langle Q_{p-s}w,w\rangle\geq-\bar C(u)\|w\|_0^2
\qquad\forall w(t,\cdot)\in H^{p-s},\quad
s=p-2,p-1\\
&&Q_0\in S^0,
\eeqsn
where $\bar C(u)$ is a positive constant which depends on $\gamma(u)$ and
on a finite number of derivatives of the fixed function $u$.
\item[{\bf Step 5.}]
We finally look at the full operator in (\ref{li}) and prove
that $e^{-\Lambda}A^{n,m}e^\Lambda$ satisfies the same estimates (\ref{stp3}) as
$e^{-\Lambda}Ae^\Lambda$, with suitable constants still depending on $\gamma$ and on
a finite number of derivatives of $u$.
Thus, the results of Step 4 hold for the full operator $(e^\Lambda)^{-1}A
e^\Lambda$ and not only for $e^{-\Lambda}Ae^\Lambda$, i.e. there exists a
constant $\bar C(u)>0$, still depending on a (higher) finite number of
derivatives of $u$, such that
\beqsn
\Re\langle (e^{\Lambda})^{-1}Ae^\Lambda w,w\rangle\geq -\bar C(u)\|w\|_0^2\qquad
\forall w(t,\cdot)\in H^{p-1}.
\eeqsn
From this, the thesis follows by the energy method.
\end{itemize}

These steps have already been followed in the proof of  Theorem \ref{thabz}
in \cite{ABZ}. Here we briefly retrace the proof of the five steps, outlining
what is new with respect to \cite{ABZ} and referring to it for all the other
computations.

\smallskip

\noindent{\bf Step 1.}
By developing asymptotically the symbols of the products of
pseudo-differential operators, we have, as in \cite{ABZ}, that:
\beqs
\nonumber
\sigma(e^{-\Lambda}Ae^\Lambda)
=&&\sum_{m=0}^{p-1}\sum_{j=m+1}^p\binom{j}{m}
(ia_j)(e^{-\Lambda}D_x^me^\Lambda)\xi^{j-m}\\
\nonumber
&&+\sum_{m=0}^{p-2}\sum_{j=m+2}^p
\sum_{\alpha=1}^{j-m-1}\sum_{\beta=0}^\alpha\frac{1}{\alpha!}
\binom{j}{m}\binom{\alpha}{\beta}(\partial_\xi^\alpha e^{-\Lambda})
(iD_x^\beta a_j)(D_x^{m+\alpha-\beta}e^\Lambda)\xi^{j-m}+A_0\\
\label{33}
=:&&A_I+A_{II}+A_0,
\eeqs
for some $A_0\in S^0$; the difference is that here $a_j=a_j(t,x,u(t,x))$
depend on
$x$ both in the second and in the third variable, so in $A_{II}$ we have to
make use of Lemma \ref{lemde} to estimate 
$$\Re(iD_x^\beta a_j)=
\Im (D_x^\beta a_j)\ {\rm and}\ |iD_x^\beta a_j|,\quad \beta\neq 0.$$

In \cite{ABZ} we have shown that the terms in $A_{II}$ with
$m+\alpha-\beta\geq1$ satisfy, for some $c>0$,
\beqs
\label{AC1}
\hspace*{7mm}
&&\hspace*{5mm}|(\partial_\xi^\alpha e^{-\Lambda})
(iD_x^\beta a_j)(D_x^{m+\alpha-\beta}
e^\Lambda)\xi^{j-m}|\\
\nonumber\hspace*{7mm}
&&\leq c|D_x^\beta a_j|
\hspace*{-8mm}
\sum_{\afrac{\afrac{\alpha_1+\ldots+\alpha_{p-1}=\alpha}{\gamma_1+
\ldots+\gamma_{p-1}=
m+\alpha-\beta}}{\afrac{r_1+\ldots+r_{q_k}=\alpha_k;\,r_i,\alpha_k
\geq1}{s_1+\ldots+s_{p_{k'}}=\gamma_{k'};\,s_i,\gamma_{k'}\geq1}}}
\prod_{k,k'=1}^{p-1}
M_{p-k}^{q_k}\frac{\langle x\rangle^{\frac{k-1}{p-1}q_k}}{
\langle \xi\rangle_h^{\alpha_k+q_k(k-1)}}\cdot
M_{p-k'}^{p_{k'}}\frac{\langle x\rangle^{\frac{k'-1}{p-1}p_{k'}-\gamma_{k'}}}{
\langle \xi\rangle_h^{p_{k'}(k'-1)}}\,\langle \xi\rangle_h^{j-m}.
\eeqs

In \eqref{AC1} the term 
$-\partial_\xi\Lambda(ia_p\xi^p)D_x\Lambda
=-r_{-1}(x,\xi)(ia_p\xi^p)$ appears. 
This term will cancel, in Step 5, with $r_{-1}(x,\xi)(ia_p\xi^p)$,
coming from $A^{1,0}$ of \eqref{li}; thus we shall omit this term in the 
following.

Each other term of (\ref{AC1}) has order
\beqsn
j-m-\alpha-\sum_{k=1}^{p-1}q_k(k-1)-\sum_{k'=1}^{p-1}p_{k'}(k'-1)\leq \min\{p-k-1,
p-k'-1\}
\eeqsn
and, in view of \eqref{euna}, decay in $x$ of the form
\beqsn
\langle x\rangle^{\frac{\sum_{k=1}^{p-1}q_k(k-1)+\sum_{k'=1}^{p-1}p_{k'}(k'-1)}{p-1}-
m-\alpha+\beta}
\leq\langle x\rangle^{-\frac{j-m-\alpha-
\sum_{k=1}^{p-1}q_k(k-1)-\sum_{k'=1}^{p-1}p_{k'}(k'-1)}{p-1}}
\eeqsn
since $-(p-1)(m+\alpha-\beta)\leq -j+m+\alpha$
for $m+\alpha-\beta\geq1$.
Thus, whenever $M_{p-k}$ or $M_{p-k'}$ appear in (\ref{AC1}),
then the order is at most $p-k-1$ and $p-k'-1$ respectively.

On the other hand, for $m+\alpha-\beta=0$, from \cite{ABZ} we have that
\beqs
\nonumber
&&\hspace*{-15mm}|\Re [(\partial_\xi^\alpha e^{-\Lambda})(iD_x^\beta a_j)
e^\Lambda\xi^{j-m}]|\\
\label{AC2}
\leq&&c |\Re(iD_x^\beta a_j)|
\sum_{\afrac{\alpha_1+\ldots+\alpha_{p-1}}{=\alpha}}\,
\prod_{k=1}^{p-1}
\sum_{\afrac{r_1+\ldots+r_{q_k}=\alpha_k}{r_i,\alpha_k\geq1}}
M_{p-k}^{q_k}\langle x\rangle^{\frac{k-1}{p-1}q_k}
\langle \xi\rangle_h^{-\alpha_k-q_k(k-1)}\langle \xi\rangle_h^{j-m},
\eeqs
for some $c>0$.

Inserting \eqref{edue} in \eqref{AC2}, 
and reminding that $D_x^\beta a_p(t)=0$ for $\beta\neq0$,
we see that each term of
(\ref{AC2}) is a symbol of order 
\beqs
\label{PK1}
j-m-\alpha
-\sum_{k=1}^{p-1}q_k(k-1)\leq p-k-1
\eeqs
with decay in $x$ of the form
\beqsn
\langle x\rangle^{\frac{\sum_{k=1}^{p-1}q_k(k-1)-j+[\beta/2]}{p-1}}
\leq\langle x\rangle^{-\frac{j-m-\alpha-\sum_{k=1}^{p-1}q_k(k-1)}{p-1}}
\eeqsn
since $[\beta/2]\leq\beta\leq\alpha+m$; hence $M_{p-k}$ appears in (\ref{AC2})
only when the order is at most $p-k-1$.
Moreover, the terms of order $p-k$ of $A_{II}$, denoted by
$\left.A_{II}\right|_{\ord(p-k)}$, all have the "right decay for the right
level" in the sense that they satisfy \eqref{stp3}.
To compute the positive
constant $C_{(M_{p-1},\ldots, M_{p-k})}$ we notice that for every $1\leq k\leq p-1$
we have that $\left.A_{II}\right|_{\ord(p-k)}$ contains derivatives of the fixed
function $u$. Let us compute the maximum number of derivatives of $u$ in
$\left.A_{II}\right|_{\ord(p-k)}$: the general term
$(\partial_\xi^\alpha e^{-\Lambda})(iD_x^\beta a_j)(D_x^{m+\alpha-\beta}e^\Lambda)
\xi^{j-m}$, with $\beta\leq\alpha$, is at level $j-m-\alpha$ because of
\eqref{27} and hence at a fixed level $j-m-\alpha=p-k$ the maximum
number of $\beta$-derivatives on $u$ appears when $\alpha\geq1$ is maximum
and hence $m=0$ and $j$ maximum, i.e. $j=p-1$ ($j=p$ is not considered because
$D_x^\beta a_p(t)=0$); in this case $p-k=j-m-\alpha=p-1-\alpha$ and the
maximum number of $\beta$-derivatives on $u$ at level $p-k$ is given by $\beta=\alpha=k-1$.
On the other hand, the minimum number of derivatives on $u$ at level $p-k$
is zero (think, for instance, at terms of the form $D_x^0a_j$) and there are
also terms which do not depend neither on $\gamma$ neither on $u$ (think at
$D_x^0a_p(t)$).
Summing up, formulas \eqref{euna}, \eqref{edue}, (\ref{AC1}),  and (\ref{AC2})
give that \beqs
\label{34}
\left|\Re \left.A_{II}\right|_{\ord(p-k)}\right|\leq
\frac{C_k(1+\gamma(u))(1+\|u\|_{1+(k-1)}^{k-1})}{\langle x\rangle^{\frac{p-k}{p-1}}}
\langle \xi\rangle_h^{p-k}
\eeqs
for some $C_k>0$, and moreover, $\Re\left.A_{II}\right|_{\ord(p-k)}$ depends only
on $M_{p-1},\ldots,M_{p-k+1}$ and not on $M_{p-k},\ldots,M_1$.

As it concerns $A_{I}$, following \cite{ABZ} it is straightforward, by means
of Lemma \ref{lemma1}, to show that it can be written as
\beqs
\label{36}
A_{I}=ia_p\xi^p+\sum_{k=1}^{p-1}(A^0_{p-k}+A^1_{p-k})+\tilde{B}_0,
\eeqs
where $$A_{p-k}^0:=ia_{p-k}\xi^{p-k} + ipa_pD_x\lambda_{p-k}\xi^{p-1}\in S^{p-k}$$
is the sum of a term $ia_{p-k}(t,x,u)\xi^{p-k}$ depending on $u$ and a term
$ipa_p(t)D_x\lambda_{p-k}(x,\xi)\xi^{p-1}$  not depending on $u$, while
$A_{p-k}^1\in S^{p-k}$ depends on the coefficients $a_j(t,x,u)$ (but not on their
derivatives) and on the symbols $\lambda_{p-1},\lambda_{p-2},\ldots,
\lambda_{p-1-(k-2)}$ (cf. \cite{ABZ}), $\tilde{B}_0\in S^0$, and from
\eqref{im}, \eqref{re} with $\beta=0$ and \eqref{a2}, \eqref{a1} we get that
\beqs
\label{38}
|\Re A^0_{p-k}(t,x,u)|+
|A^1_{p-k}(t,x,u)|\leq
\frac{C_k (1+\gamma(u))\langle \xi\rangle_h^{p-k}}{\langle x\rangle^{\frac{p-k}{p-1}}}
\eeqs
possibly enlarging the constant $C_k$ of \eqref{34}, $1\leq k\leq p-1$;
moreover, $A^0_{p-k}$ depends only on $M_{p-k}$ and $A^1_{p-k}$ depends only on
$M_{p-1},\ldots,M_{p-k+1}$.
Step 1 is so completed.

\smallskip

\noindent{\bf Step 2.}
We now look at the real part of the terms $A_{p-k}$ of order $p-k$ in \eqref{33}:
\beqs
\nonumber
A_{p-k}:=&&
\left.A_{I}\right|_{\ord(p-k)}+\left.A_{II}\right|_{\ord(p-k)}\\
\label{227bis}
=&&A^0_{p-k}+A^1_{p-k}+\left.A_{II}\right|_{\ord(p-k)},
\qquad k=1,\ldots,p-1.
\eeqs

From the definition \eqref{26} of $\lambda_{p-k}$, conditions \eqref{ap} and
(\ref{im}) with $\beta=0$, estimates (\ref{34}) and (\ref{38}), for
$|\xi|\geq2h$ we have:
\beqs
\nonumber
\Re A_{p-k}=&&\Re(ipa_pD_x\lambda_{p-k}\xi^{p-1}+ia_{p-k}\xi^{p-k})+
\Re(A_{p-k}^1)+\Re(\left.A_{II}\right|_{\ord(p-k)})
\\
\nonumber
&&=pa_p\xi^{p-1}M_{p-k}\frac{|\xi|^{p-1}}{\xi^{p-1}}
\langle x\rangle^{-\frac{p-k}{p-1}}
\psi\left(\frac{\langle x\rangle}{\langle \xi\rangle_h^{p-1}}\right)
\langle \xi\rangle_h^{-k+1}-
\Im a_{p-k}\cdot\xi^{p-k}
\\
\nonumber
&&+\Re(A_{p-k}^1)+\Re(\left.A_{II}\right|_{\ord(p-k)})
\\
\nonumber
\geq&&\left(\frac{2}{\sqrt{5}}\right)^{p-1}
pC_p\frac{M_{p-k}}{\langle x\rangle^{\frac{p-k}{p-1}}}
\langle \xi\rangle_h^{p-k}\psi-
\frac{C\gamma(u)}{\langle x\rangle^{\frac{p-k}{p-1}}}\langle \xi\rangle_h^{p-k}\psi-
\frac{C\gamma(u)}{\langle x\rangle^{\frac{p-k}{p-1}}}\langle \xi\rangle_h^{p-k}
(1-\psi)
\\
\nonumber
&&-C_k(1+1+\|u\|_{1+(k-1)}^{k-1})(1+\gamma(u))
\frac{\langle\xi\rangle_h^{p-k}}{\langle
x\rangle^{\frac{p-k}{p-1}}}
\\
\label{39}
\geq&&\psi\left(\frac{\langle x\rangle}{\langle \xi\rangle_h^{p-1}}\right)
\cdot\left[\left(\frac{2}{\sqrt{5}}\right)^{p-1}pC_pM_{p-k}-
  \left(C+C_k(2+\|u\|_{1+(k-1)}^{k-1})\right)(1+\gamma(u))\right]
\frac{\langle \xi\rangle_h^{p-k}}{\langle x\rangle^{\frac{p-k}{p-1}}}
\\
&&
\nonumber
-\tilde C_k (1+\gamma(u))(1+\|u\|_{1+(k-1)}^{k-1})
\eeqs
for some $\tilde C_k>0$ since
$|\xi|
\geq\frac{2}{\sqrt{5}}\langle\xi\rangle_h$
and $\langle \xi\rangle_h^{p-1}/
\langle x\rangle$ is bounded on the support of $(1-\psi)$.
Notice that the constants $C_k,\tilde C_k$ depend only on
$M_{p-1},\ldots,M_{p-k+1}$
and not on $M_{p-k}$, and that with a new constant $C_k'>0$ we can write
\beqs\label{equesta}\quad
\Re A_{p-k}&\geq& \psi\left(\frac{\langle x\rangle}{\langle \xi\rangle_h^{p-1}}
\right)\hskip-0.2cm\left[\left(\frac{2}{\sqrt{5}}\right)^{p-1}pC_pM_{p-k}-
  C_k'(1+\|u\|_{1+(k-1)}^{k-1})(1+\gamma(u))\right]
\frac{\langle \xi\rangle_h^{p-k}}{\langle x\rangle^{\frac{p-k}{p-1}}}
\\
\nonumber &&-\tilde C_k (1+\gamma(u))(1+\|u\|_{1+(k-1)}^{k-1}).
\eeqs

In particular, for $k=1$,
\beqsn
\Re A_{p-1}
\geq \psi\left(\frac{\langle x\rangle}{\langle \xi\rangle_h^{p-1}}\right)
\cdot\left[\left(\frac{2}{\sqrt{5}}\right)^{p-1}pC_pM_{p-1}-2C_1'(1+\gamma(u))
  \right]
\frac{\langle \xi\rangle_h^{p-1}}{\langle x\rangle}-2\tilde C_1 (1+\gamma(u)),
\eeqsn
and we can choose $M_{p-1}>0$ sufficiently
large, i.e.
\[M_{p-1}=M_{p-1}(u)\geq \ds\frac{C_1'(1+\gamma(u))}{2^{p-2}/\sqrt{5}^{p-1}pC_p}\]
so that
\beqsn
\Re A_{p-1}(t,x,u,\xi)\geq -2\tilde C_1 (1+\gamma(u))
\qquad\forall(t,x,\xi)\in[0,T]\times\R^2.
\eeqsn
Applying the sharp-G{\aa}rding Theorem~\ref{thA1} to
$A_{p-1}+2\tilde C_1 (1+\gamma(u))$ we can thus find
pseudo-differential operators $Q_{p-1}(t,x,u,D_x)$ and $R_{p-1}(t,x,u,D_x)$
with symbols $Q_{p-1}(t,x,u,\xi)\in S^{p-1}$ and $R_{p-1}(t,x,u,\xi)\in S^{p-2}$
such that
\beqs
\label{nsg1}
&&A_{p-1}(t,x,u,D_x)=Q_{p-1}(t,x,u,D_x)+R_{p-1}(t,x,u,D_x)-2\tilde C_1 (1+\gamma(u))
\eeqs
where
$$\Re\langle Q_{p-1}(t,x,u,D_x)w(t,x),w(t,x)\rangle\geq0
\qquad\forall (t,x)\in[0,T]\times\R,\ \forall w(t,\cdot)\in H^{p-1}(\R),$$
and $R_{p-1}$ has the asymptotic development given in \eqref{A1} for $A_{p-1}$.
From (\ref{33}), (\ref{36}), \eqref{227bis} and (\ref{nsg1}) we get:
\beqs
\nonumber
\sigma(e^{-\Lambda} Ae^\Lambda)
=&&ia_p\xi^p+A_{p-1}
+\sum_{k=2}^{p-1}A_{p-k}+A'_0\\
\label{47}
=&&ia_p\xi^p+Q_{p-1}
+\sum_{k=2}^{p-1}(\left.A_{I}\right|_{\ord(p-k)}+\left.A_{II}\right|_{\ord(p-k)}+
\left.R_{p-1}\right|_{\ord(p-k)})+A''_0
\eeqs
for some $A'_0,A''_0\in S^0$,
where $\left.R_{p-1}\right|_{\ord(p-k)}$ denotes the terms of order $p-k$ of
$R_{p-1}$, for $2\leq k\leq p$. To complete step 2, we need to check that the terms
$\left.R_{p-1}\right|_{\ord(p-k)}$ satisfy \eqref{stp3} with a new constant
$C_{(M_{p-1},\ldots, M_{p-k})}(u)$.
In \cite{ABZ} it has already been proved, using Lemma \ref{lemma1} and
conditions \eqref{3}, \eqref{4} instead of \eqref{im}, \eqref{re},
\eqref{dew}, that the terms of $\left.R_{p-1}\right|_{\ord(p-k)}$ all have the
right decay for the right level $p-k$; here we only need to
find by means of Lemma \ref{lemde} the precise constant $C_{M_{p-1}}(u)$
(note that the constant depends only on $M_{p-1}$ since $R_{p-1}$ depends only on $A_{p-1}$). The term $A_{p-1}$ does not contain derivatives of $u$,
since in order that a term of type
$(\partial_\xi^\alpha e^{-\Lambda})(iD_x^\beta a_j)(D_x^{m+\alpha-\beta}e^\Lambda)
\xi^{j-m}$ of \eqref{33} is at level $j-m-\alpha=p-1$ with $\alpha\geq1$ we must have
$j=p$. Therefore, looking at the asymptotic development of $R_{p-1}$ given by Remark \ref{remA} with $m=p-1$ and $\ell=k-1$, the maximum number of derivatives on $u$ in
$\left.R_{p-1}\right|_{\ord(p-k)}$ appears in \eqref{Asc}  when $\alpha=0$ and $\beta=2\ell+1=2k-1$.

By Lemma \ref{lemde} we
come so to the estimate
\beqs\label{tu}
|R_{p-1}(t,x,u,\xi)\vert_{\ord(p-k)}|\leq \frac{C_k(1+\gamma(u))
  (1+\|u\|_{1+(2k-1)}^{2k-1})}{\langle x\rangle^{\frac{p-k}{p-1}}}
\langle\xi\rangle^{p-k}_h,
\eeqs
possibly enlarging the constant $C_k$ of  \eqref{34}, \eqref{38}.
Step 2 is completed.

\smallskip

\noindent{\bf Step 3.} In order to reapply sharp-G{\aa}rding Theorem~\ref{thA1}
we now have to investigate the action of that theorem
to each term of the form
$\left.A_{I}\right|_{\ord(p-k)}+\left.A_{II}\right|_{\ord(p-k)}+S_{p-k}$, where
$S_{p-k}$ denotes terms of order $p-k$ coming from remainders of
previous applications of the sharp-G{\aa}rding Theorem, for $p-k\geq3$.
In \cite{ABZ} we have computed and estimated the generic remainder
$$R(\left.A_{I}\right|_{\ord(p-k)})+R(\left.A_{II}\right|_{\ord(p-k)})+R(S_{p-k})$$
under the assumptions of Theorem \ref{thabz}, showing that it is sum of terms
of order $p-j$, $k+1\leq j\leq p$, each one of them with the right decay
$(p-j)/(p-1)$ and the right constants $M_{p-1},\ldots, M_{p-k+1}$ for the right level $p-j$. Here, we can argue with the
same (quite long and technical) computations and make use of Lemma
\ref{lemde} instead of assumptions \eqref{3} and \eqref{4} to get that this
generic remainder consists in a sum of terms of order
$p-j$, $k+1\leq j\leq p$, each one of them satisfying \eqref{stp3}.
It only remains to compute precisely the corresponding constant
$C_{(M_{p-1}, \ldots, M_{p-k+1})}(u)$.

To this aim, we need to better understand the dependence of $S_{p-k}$ on $u$;
let us first focus on the second application of the sharp-G{\aa}rding
Theorem~\ref{thA1}.
From (\ref{47}) with $R_{p-1}=R(A_{p-1})$ we have
\beqsn
\sigma(e^{-\Lambda}Ae^\Lambda)=ia_p\xi^{p}+Q_{p-1}+A_{p-2}
+R(A_{p-1})\big|_{\ord(p-2)}
+\sum_{k=3}^{p-1}(A_{p-k}+R(A_{p-1})\big|_{\ord(p-k)})
+A''_0.
\eeqsn

Since from \eqref{tu} $R(A_{p-1})\big|_{\ord(p-2)}$ has the same structure as
$A_{p-2}$, depends on the same constant, and bears $2k-1=3$ derivatives of
$u$ (much more than $A_{p-2}$, see \eqref{34} and \eqref{38}), we can follow
the computations in \eqref{39} to get, instead of \eqref{equesta},
\beqsn
&&\hskip-2cm\Re\left(A_{p-2}+\left.R(A_{p-1})\right|_{\ord(p-2)}
\right)(t,x,u,\xi)\geq
\\
&&\hskip+1cm\geq \psi\hskip-0.1cm\left(\frac{\langle x\rangle}{\langle
  \xi\rangle_h^{p-1}}\right)\hskip-0.15cm\left[
  \left(\frac{2}{\sqrt{5}}\right)^{p-1}\hskip-0.6cmpC_pM_{p-2}-
  \hskip-0.1cmC_2'\left(1+\|u\|_{4}^{3}\right)(1+\gamma(u))\right]\hskip-0.1cm
\frac{\langle \xi\rangle_h^{p-2}}{\langle x\rangle^{\frac{p-2}{p-1}}}
\\
&&\hskip+1,4cm-\tilde C_2(1+\gamma(u))(1+\|u\|_{4}^{3})
\eeqsn
for some $C'_2,\tilde{C}_2>0$.

We can so choose $M_{p-2}>0$ sufficiently
large, i.e.
\[M_{p-2}\geq \ds\frac{C_2'(1+\|u\|_{4}^{3})(1+\gamma(u))}{2^{p-1}/
  \sqrt{5}^{p-1}pC_p}\]
so that
\beqsn
\Re\left(A_{p-2}+\left.R(A_{p-1})\right|_{\ord(p-2)}
\right)(t,x,u,\xi)\geq -\tilde C_2(1+\gamma(u))(1+\|u\|_{4}^{3}),
\quad\forall(t,x,\xi)\in[0,T]\times\R^2
\eeqsn
and we can apply the sharp-G{\aa}rding Theorem~\ref{thA1} to
$$A_{p-2}+\left.R(A_{p-1})\right|_{\ord(p-2)}+\tilde C_2(1+\gamma(u))
(1+\|u\|_{4}^{3})$$ obtaining that there exist
pseudo-differential operators $Q_{p-2}$ and $R_{p-2}$, with symbols in
$S^{p-2}$ and $S^{p-3}$ respectively, such that
\beqsn
&&\Re\langle Q_{p-2}w,w\rangle\geq0\qquad\forall w(t,\cdot)\in H^{p-2}\\
&&A_{p-2}+R(A_{p-1})\big|_{\ord(p-2)}=Q_{p-2}
+R_{p-2}-\tilde C_2(1+\gamma(u))(1+\|u\|_{4}^{3}),
\eeqsn
with
\beqsn
R_{p-2}=R(A_{p-2}+
R(A_{p-1})\big|_{\ord(p-2)})
=R(A_{p-2})+R(R(A_{p-1})\big|_{\ord(p-2)}),
\eeqsn
so that with the notation $R^2(A_{p-1})=R(R(A_{p-1}))$
\beqsn
\sigma(e^{-\Lambda}Ae^\Lambda)=&&ia_p\xi^p+Q_{p-1}+Q_{p-2}\\
&&+A_{p-3}+R(A_{p-1})\big|_{\ord(p-3)}
+R(A_{p-2})\big|_{\ord(p-3)}+R^2(A_{p-1})\big|_{\ord(p-3)}\\
&&+\sum_{k=4}^{p-1}\left(A_{p-k}+R(A_{p-1})\big|_{\ord(p-k)}
+R(A_{p-2})_{\ord(p-k)}+R^2(A_{p-1})\big|_{\ord(p-k)}\right)+A'_0.
\eeqsn
At the second application of sharp-G{\aa}rding Theorem~\ref{thA1} the term
$R^2(A_{p-1})$ appears at (maximum) level $p-3$.
By \eqref{tu} we know that $R(A_{p-1})\big|_{\ord(p-2)}$ contains
$2k-1=3$ derivatives on $u$; so, its remainder, given by  \eqref{A1}, has the structure of Remark \ref{remA} and by \eqref{Asc} with $m=p-2$ and $\ell=k-2$ we see that the maximum number of derivatives with respect to $u$ in $R(R(A_{p-1})\big|_{\ord(p-2)})\big|_{\ord(p-k)}$ appears when $\beta=2\ell+1=2k-3$ and is given by $3+(2k-3)=2k$.

Analogously, by formula
\eqref{A1} and Remark \ref{remA}, we have that $R(A_{p-2})$ consists of terms
$R(A_{p-2})\big|_{\ord(p-k)}$ with $3\leq k\leq p$; formula \eqref{Asc} with
$m=p-2$ and $\ell=k-2$, together with \eqref{34} and \eqref{38}, give that
$R(A_{p-2})\big|_{\ord(p-k)}$ contains at most $1+(2k-3)=2k-2$ derivatives on $u$.

Summing up, the maximum number of derivatives of $u$ appears in $R^2(A_{p-1})$
and we get:
\beqs\nonumber
\left|\left(R(A_{p-1})\big|_{\ord(p-k)}
+R(A_{p-2})_{\ord(p-k)}\right.\right.&&\left.\left.+R^2(A_{p-1})\big|_{\ord(p-k)}\right)(t,x,u,\xi)\right|\leq
\\
\label{tuu}
&&\leq \frac{C_k(1+ \gamma(u))(1+\|u\|_{1+2k}^{2k})}{\langle x
  \rangle^{\frac{p-k}{p-1}}}\langle\xi\rangle_h^{p-k},\ 3\leq k\leq p.
\eeqs
Now, let us come back to the general case.

At the $j-th$ application of sharp-G{\aa}rding Theorem we find, at level
$p-j$, the terms
\beqsn
A_{p-j}+\ds\sum_{\afrac{1\leq k\leq j-1}{1\leq s\leq j-k}}R^s (A_{p-k})\vert_{p-j}.
\eeqsn
These terms depend on $u$ and its derivatives; reminding that
$$A_{p-j}=A_{I}\vert_{\ord(p-j)}+A_{II}\vert_{\ord(p-j)}=ia_{p-j}\xi^{p-j}+ipa_p
D_x\lambda_{p-j}\xi^{p-1}+A^1_{p-k}+A_{II}\vert_{\ord(p-j)},$$ see also \eqref{34}
and \eqref{38}, we see that the maximum number of derivatives of $u$ is in
the term $R^{j-1} (A_{p-1})\vert_{\ord(p-j)}, $ i.e. in the principal part of $R^{j-1} (A_{p-1})$.
To compute this number, we work by induction. For $j=2$, by \eqref{tu} with $k=2$ we know that $R^{j-1} (A_{p-1})\vert_{\ord(p-j)}$ contains $2k-1=3$ derivatives of $u$; for $j=3$, by \eqref{tuu} with $k=3$ we know that $R^{j-1} (A_{p-1})\vert_{\ord(p-j)}$ contains $2k=6$ derivatives of $u$; let us now suppose that for all $2\leq s\leq j-1$ we have that $R^{s-1} (A_{p-1})\vert_{\ord(p-s)}$ contains $3(s-1)$ derivatives of $u$, and prove that $R^{j-1} (A_{p-1})\vert_{\ord(p-j)}$ contains $3(j-1)$ derivatives of $u$. Arguing as for $R(A_{p-1})$ and $R^2(A_{p-1})$ we obtain that the remainder $R^{j-1} (A_{p-1})=R(R^{j-2}(A_{p-1}))$ is the sharp-Garding remainder of the operator $R^{j-2}(A_{p-1})$ with symbol of order $p-1-(j-2)=p-j+1$; the principal part of $R(R^{j-2}(A_{p-1}))$ consists so in a term of order $p-j$, depending on $u$ and its derivatives, and given by \eqref{Asc} with $m=p-j+1$ and $\ell=1$. The maximum number of derivatives on $u$ appears so when $\alpha=0$ and $\beta=3$ in \eqref{Asc} and is given, by the inductive hypothesis, by $3(j-2)+3=3(j-1)$.

It follows that
\beqsn
|R^{j-1} (A_{p-1})\vert_{\ord(p-j)}(t,x,u,\xi)|\leq \frac{C_j(1+ \gamma(u))
  (1+\|u\|_{1+3(j-1)}^{3(j-1)})}{\langle x\rangle^{\frac{p-j}{p-1}}}\langle\xi
\rangle_h^{p-j},\ 2\leq j\leq p.
\eeqsn
Thus, at each level $p-j$, $2\leq j\leq p$, we have
\beqsn
\left\vert A^1_{p-j }+A_{II}\vert_{\ord(p-j)}+\ds\sum_{\afrac{1\leq k\leq j-1}{1\leq s\leq j-k}}R^s (A_{p-k})\vert_{p-j}\right\vert(t,x,u,\xi)
\leq\frac{C_j(1+ \gamma(u))(1+\|u\|_{1+3(j-1)}^{3(j-1)})}{\langle x
  \rangle^{\frac{p-j}{p-1}}}\langle\xi\rangle_h^{p-j},
\eeqsn
with $C_j$ depending on $M_{p-1},\ldots, M_{p-j+1}$ and not on
$M_{p-j},\ldots, M_1.$

Thanks to the estimates given here above, we can apply again and again the
sharp-G{\aa}rding Theorem~\ref{thA1} to find pseudo-differential
operators $Q_{p-1},Q_{p-2},\ldots,Q_3$ of order $p-1,p-2,\ldots,3$
respectively and all positive definite, such that
\beqsn
\sigma(e^{-\Lambda}Ae^\Lambda)=ia_p\xi^p
+Q_{p-1}+Q_{p-2}+\ldots +Q_3
+(A_{2}+S_{2})+(A_{1}+S_{1})+(A_0+S_0),
\eeqsn
with $S_j,$ $j=0,1,2$ coming from remainders of the sharp-G{\aa}rding theorem,
and such that
\beqs
\label{Aj}\quad
|(A_j+S_j)(t,x,u,\xi)|\leq C_{
  p-j}\cdot (1+\gamma(u))(1+\|u\|_{1+3(p-j-1)}^{3(p-j-1)})\langle x
\rangle^{-\frac{j}{p-1}}\langle\xi\rangle_h^{j},\; j=0,1,2,
\eeqs
with $C_{p-j}>0$ depending on $M_{p-1},\ldots,M_{j+1}$ and not on $M_j,\ldots,M_1.$

\smallskip

\noindent{\bf Step 4.}
Let us split the term of order 2 into $\Re(A_2+S_2)+i\Im(A_2+S_2)$; by
\eqref{equesta} and the discussion of step 3 we have that
\beqsn
\Re(A_2+S_2)&\geq&
\psi\left(\frac{\langle x\rangle}{\langle \xi\rangle_h^{p-1}}\right)
\cdot
\left[
  \left(\frac{2}{\sqrt{5}}\right)^{p-1}\hskip-0.5cmpC_pM_{2}-C_{p-2}''
  (1+\|u\|_{1+3(p-3)}^{3(p-3)})(1+\gamma(u))
\right]
\langle x\rangle^{-\frac{2}{p-1}}\langle \xi\rangle_h^2
\\
&&-\bar C_{p-2}(1+\|u\|_{1+3(p-3)}^{3(p-3)})(1+\gamma(u)),
\eeqsn
for some $C''_{p-2},\bar{C}_{p-2}>0$
and we can choose $M_2$ large enough so that
\beqsn
\Re(A_2+S_2)\geq -\bar C_{p-2}(1+\gamma(u))(1+\|u\|^{3(p-3)}_{1+3(p-3)}).
\eeqsn
Then, the Fefferman-Phong inequality~\eqref{A3}
applied to
\beqsn
\tilde{A}_2(t,x,u(t,x),\gamma(u(t,x)),D_x):=\Re(A_2+S_2)+\bar C_{p-2}
(1+\gamma(u))(1+\|u\|^{3(p-3)}_{1+3(p-3)})
\eeqsn
gives
\beqsn
\Re\langle (\Re(A_2+S_2))w,w\rangle
\geq-c(u)(1+\gamma(u))(1+\|u\|^{3(p-3)}_{1+3(p-3)})\|w\|_0^2
\eeqsn
without any remainder, for a new constant $c(u)>0$
depending on the derivatives
$\partial_\xi^\alpha \partial_x^\beta$ of the symbol of  $\tilde{A}_2$
with $|\alpha|+|\beta|\leq 7$,
by Remark \ref{remA4}. Being the function $\gamma$ of class $C^7$ by assumption
we can then find a constant $C_\gamma>0$, depending only on $\gamma$, such that
\beqs
\nonumber
\Re\langle (\Re(A_2+S_2))w,w\rangle
\geq&&-C_\gamma(1+\|u\|^{1+3(p-3)+7}_{1+3(p-3)+7})\|w\|_0^2\\
\label{tt2}
=&&-C_\gamma(1+\|u\|^{3p-1}_{3p-1})\|w\|_0^2.
\eeqs

On the other hand, we split $i\Im(A_2+S_2)$ into its hermitian and
antihermitian part:
\beqsn
\frac{i\Im(A_2+S_2)+(i\Im(A_2+S_2))^*}{2}+\frac{i\Im(A_2+S_2)-
  (i\Im(A_2+S_2))^*}{2},
\eeqsn
and we have that $\Re\langle\frac{i\Im(A_2+S_2)-(i\Im(A_2+S_2))^*}{2}w,
w\rangle=0,$ while $i\Im(A_2+S_2)+(i\Im(A_2+S_2))^*$ can be put together with
$A_1+S_1$ since it has a real principal part of order 1, does not depend on
$M_1$, and has the ``right
decay'' for level 1.
Therefore we can choose $M_1>0$ sufficiently large so that, by \eqref{Aj},
\beqsn
\Re\left(A_1+S_1+\frac{i\Im(A_2+S_2)+(i\Im(A_2+S_2))^*}{2}\right)\geq -
\bar C_{p-1}(1+\gamma(u))(1+\|u\|^{3(p-2)}_{1+3(p-2)}),
\eeqsn
for some $\bar{C}_{p-1}>0$
and hence, by the
sharp-G{\aa}rding inequality~\eqref{A2} we get
\beqs
\nonumber
\Re\langle \big(A_1+S_1+&&
\frac{i\Im(A_2+S_2)+(i\Im(A_2+S_2))^*}{2}\big)
w,w\rangle\geq
\\
\label{t2}
&&\geq  -C'_\gamma(1+\|u\|^{1+3(p-2)}_{1+3(p-2)})
\|w\|_0^2,
\eeqs
for a new constant $C'_\gamma>0$.

Summing up, throughout steps 1-4 we have obtained
\beqsn
\sigma(e^{-\Lambda}Ae^\Lambda)=ia_p\xi^p+
\sum_{j=1}^{p-3}Q_{p-j}
+(A_2+S_2)+(A_1+S_1)+(A_0+S_0)
\eeqsn
with
\beqsn
&&\Re\langle Q_{p-j}w,w\rangle\geq0\qquad\forall w(t,\cdot)\in H^{p-j},\
1\leq j\leq p-3\\
&&\Re\langle(A_2+S_2+A_1+S_1)w,w\rangle\geq -\tilde{C}_\gamma
(1+\|u\|^{3p-1}_{3p-1})\|w\|_0^2
\qquad\forall w(t,\cdot)\in H^2,
\eeqsn
for a positive constant $\tilde{C}_\gamma$, because of \eqref{tt2} and
\eqref{t2}, since $3p-1>1+3(p-2)$.

\smallskip

\noindent{\bf Step 5.} Now, we come back to (\ref{li}), and
remark that $A^{n,m}$ is of the same
kind of $A$ with $D_x^m a_j$ instead of $a_j$ and
$(\partial_\xi^mr^n)\xi^j$ instead of $\xi^j$, with $0\leq m\leq p-2$ and
$1\leq n\leq p-1-m$.
This implies that we have $m$ more $x$-derivatives on $a_j$, but
the level in $\xi$ decreases of $-n-m<-m$, so we still have the right
decay for the right level and the right dependence on the constants
$M_{p-j}$, $j\leq k$, at each level $p-k$. As far as the derivatives of the
fixed function $u$ are concerned, the maximum number of derivatives of $u$ appears in the
term $\partial_\xi^m r^nD_x^m(ia_j(t,x,u)\xi^j)\in S^{j-m-n}$ with $j=p-1$,
$n=1$, $m=p-2$, so that we argue as for
$\sigma(e^{-\Lambda}Ae^\Lambda)$ and find that also
\beqsn
\sigma(e^{-\Lambda}A^{n,m}e^\Lambda)=\sum_{s=0}^pQ_{p-s}^{n,m}
\eeqsn
with $Q_0^{n,m}\in S^0$ and
\beqsn
\Re\langle Q_{p-s}^{n,m}w,w\rangle\geq&&-C^{n,m}_\gamma
(1+\|u\|^{3p-1+p-2}_{3p-1+p-2})\|w\|_0^2\\
=&&-C^{n,m}_\gamma
(1+\|u\|^{4p-3}_{4p-3})\|w\|_0^2
\qquad\forall w(t,\cdot)\in H^{p-s},\ 1\leq s\leq p-1
\eeqsn
for some $C^{n,m}_\gamma>0$.

Summing up, we have proved that
\beqs\label{sumup}
\Re\langle (e^\Lambda)^{-1}Ae^\Lambda w,w\rangle\geq-c_\gamma
(1+\|u\|^{4p-3}_{4p-3}) \|w\|_0^2\qquad
\forall w(t,\cdot)\in H^{p}
\eeqs
for some $c>0$.
From \eqref{ipilambda} and \eqref{sumup} it follows that every
$w\in C([0,T]; H^{p})\cap C^1([0,T]; H^{0})$ satisfies:
\beqsn
\frac{d}{dt}\|w\|_0^2=&&2\Re\langle\partial_tw,w\rangle=2\Re\langle
iP_\Lambda w,w\rangle-2\Re\langle (e^\Lambda)^{-1}Ae^\Lambda w,w\rangle
\\
\leq&&\|P_\Lambda w\|_0^2+\|w\|_0^2-2
\Re\langle (e^\Lambda)^{-1}Ae^\Lambda w,w\rangle\\
\leq&&c'_\gamma(1+\|u\|^{4p-3}_{4p-3})(\|P_\Lambda w\|_0^2+\|w\|_0^2)
\eeqsn
for some $c'_\gamma>0$.
Applying Gronwall's inequality, we deduce that for all $w\in C([0,T]; H^{p})\cap C^1([0,T]; H^{0})$, the estimate
\beqsn
\hskip+0.5cm\|w(t,\cdot)\|_0^2\leq e^{c'_{\gamma}(1+\|u\|^{4p-3}_{4p-3})t}
\left(\|w(0,\cdot)\|_0^2+
\int_0^t\|P_\Lambda w(\tau,\cdot)\|_0^2d\tau\right),\;\forall t\in[0,T]
\eeqsn
holds.
Since $\langle D_x\rangle^sP_\Lambda \langle D_x\rangle^{-s}$ satisfies, for every $s\in\R$, the same hypotheses as $P_\Lambda$, we immediately get that for every $s\in\R$, $w\in C([0,T]; H^{s+p})\cap C^1([0,T]; H^{s})$ we have
\beqs
\label{E3'}
\hskip+0.5cm\|w(t,\cdot)\|_s^2\leq e^{C_{s,\gamma}(1+\|u\|^{4p-3}_{4p-3})t}
\left(\|w(0,\cdot)\|_s^2+
\int_0^t\|P_\Lambda w(\tau,\cdot)\|_s^2d\tau\right),\;\forall t\in[0,T]
\eeqs
for a positive constant $C_{s,\gamma}.$

The a priori estimate \eqref{E3'} implies, by standard arguments from the energy method, that for every $u_{0,\Lambda}\in H^s$ and $f_\Lambda\in
C([0,T], H^s)$ the Cauchy problem (\ref{6}) has a unique solution
$w\in C([0,T]; H^s)$ satisfying
\beqs
\label{E3}
\|w(t,\cdot)\|_s^2\leq e^{C_{s,\gamma}(1+\|u\|^{4p-3}_{4p-3})t}\left(\|u_{0,\Lambda}\|_s^2+
\int_0^t\|f_\Lambda(\tau,\cdot)\|_s^2d\tau\right)
\qquad\forall t\in[0,T].
\eeqs
Since $e^\Lambda\in S^\delta$, for $v=e^\Lambda w$ we finally have,
from (\ref{E3}) with $s-\delta$ instead of $s$:
\beqsn
\|v\|_{s-2\delta}^2
\leq&&c_1\|w\|_{s-\delta}^2\leq c_2e^{C_{s,\gamma}(1+\|u\|^{4p-3}_{4p-3})t}
\left(\|u_{0,\Lambda}\|_{s-\delta}^2+
\int_0^t\|f_\Lambda\|_{s-\delta}^2d\tau\right)\\
\leq&&C'_{s,\gamma}e^{(1+\|u\|^{4p-3}_{4p-3})t}
\left(\|u_0\|_s^2+\int_0^t\|f\|_s^2d\tau\right)
\eeqsn
for some $c_1,c_2,C'_{s,\gamma}>0$, that is \eqref{Eoldv}.
This proves Theorem \ref{iop}. In particular, we have that for every $f\in C([0,T]; H^\infty)$ and $u_0\in H^\infty$ there
exists a solution $v\in C([0,T];H^\infty(\R))$ of (\ref{2lin}) which
satisfies an energy estimate of the form (\ref{Eold}) with constant
$$C_s=C_s(u)=C'_{s,\gamma}e^{(1+\|u\|^{4p-3}_{4p-3})t}.$$
\end{proof}

\section{The semilinear problem}
\label{sec4}

In this section we consider the semilinear Cauchy problem \eqref{2} and give the proof of the main result of this paper, Theorem \ref{th1}.

We set $X:=C^1([0,T];H^\infty(\R))$ and consider the map
\beqsn
T:\ X&&\longrightarrow X\\
u&&\longmapsto T(u)
\eeqsn
defined by
\beqs\label{heart}
T(u):=&&u(t,x)-u_0(x)+i\int_0^ta_p(s)D_x^pu(s,x)ds\\
\nonumber
&&+i\sum_{j=0}^{p-1}\int_0^ta_j(s,x,u(s,x))D_x^ju(s,x)ds-i\int_0^tf(s,x)ds.
\eeqs

\begin{Rem}\label{remE}
The existence of a local solution
$u\in C^1([0,T^*];H^\infty(\R))$ of the Cauchy problem \eqref{2} is
equivalent to the existence of a solution $u\in C^1([0,T^*];H^\infty(\R))$ of
\beqs
\label{4bis}
T(u)\equiv0\quad\mbox{in}\ [0,T^*]\times\R.
\eeqs
Indeed, if $T(u)=0$ then
\beqs
\nonumber
u(t,x)=&&u_0(x)-i\int_0^ta_p(s)D_x^pu(s,x)ds\\
\label{2bis}
&&-i\sum_{j=0}^{p-1}\int_0^ta_j(s,x,u(s,x))D_x^ju(s,x)ds+i\int_0^tf(s,x)ds
\eeqs
and hence $u(0,x)=u_0(x)$ and
\beqs
\label{3bis}
D_tu(t,x)=&&-i\partial_tu(t,x)=-a_p(t)D_x^pu(t,x)-\sum_{j=0}^{p-1}a_j(t,x,u(t,x))
D_x^ju(t,x)+f(t,x)
\eeqs
so that $u$ solves \eqref{2}.
Vice versa, if $u\in C^1([0,T^*];H^\infty(\R))$ is a solution of the
Cauchy problem \eqref{2}, then integrating \eqref{3bis} with respect to time
we get \eqref{2bis}, i.e. $T(u)=0$.
\end{Rem}

We are so reduced to prove the
existence of $0<T^*\leq T$ and of a unique solution $u\in
C^1([0,T^*];H^\infty(\R))$ of \eqref{4bis}.

To this aim we shall use the Nash-Moser Theorem~\ref{thNM}.
Note that $X=C^1([0,T];H^\infty(\R))$, with the family of semi-norms
\beqsn
\3|g\3|_n=\sup_{[0,T]}\left(\|g(t,\cdot)\|_n+\|D_tg(t,\cdot)\|_n\right),
\qquad n\in\N_0,
\eeqsn
is a tame space, see Example~\ref{proptame} in Appendix \ref{app:a}.

The map $T$ is smooth tame, since it is a composition of linear and nonlinear
operators and of integrations, which are all smooth tame by
Remark~\ref{remA1}, and since it does not contain time derivatives
(this is important since $\partial_t$ does not operate from $X$ to $X$, so
it cannot be a tame map).

In order to apply the Nash-Moser Theorem~\ref{thNM}, denoting by $DT(u)v$ the
Fr\'echet derivative of $T$ at $u$ in the direction $v$, we shall
prove that the
equation
$DT(u)v=h$ has a unique solution $v:=S(u,h)$ for all $u,h\in X$ and
that $S:\ X\times X\to X$ is smooth tame. This is going to be done in the
following lemmas.
\begin{Lemma}
For every $u,h\in X$, the equation $DT(u)v=h$ admits a unique solution $v\in X$, and the solution satisfies for every $n\in\N_0$ the following estimate:
\beqs
\label{6bis}
\|v(t,\cdot)\|_n^2\leq C_n(u)\left(\|h(0,\cdot)\|_{n+r}^2+\int_0^t
\|D_th(\tau,\cdot)\|_{n+r}^2d\tau\right)
\qquad\forall t\in[0,T],
\eeqs
for any $r\geq\sigma$,
with $C_n(u):=C_{n+\sigma,\gamma}\exp\big\{\big(1+\|u\|^{4p-3}_{4p-3}
\big)T\big\}$ as in \eqref{Eoldv}.
\end{Lemma}
\begin{proof}
Let us compute by the definition \eqref{heart} of the map $T$, the Fr\'echet derivative of $T$, for $u,v\in X$:
\beqsn
DT(u)v=&&\lim_{\varepsilon\to0}\frac{T(u+\varepsilon v)-T(u)}{\varepsilon}\\
=&&\lim_{\varepsilon\to0}\Big\{v+i\int_0^ta_p(s)D_x^pv(s)ds
+i\sum_{j=0}^{p-1}\int_0^t
\frac{a_j(s,x,u+\varepsilon v)-a_j(s,x,u)}{\varepsilon}D_x^ju(s)ds\\
&&+i\sum_{j=0}^{p-1}\int_0^ta_j(s,x,u+\varepsilon v)D_x^jv(s)ds\Big\}\\
=&&v+i\int_0^t\hskip-0.25cma_p(s)D_x^pv(s)ds+i\sum_{j=0}^{p-1}\int_0^t\hskip-0.25cm
\partial_wa_j(s,x,u)v(s)D_x^ju(s)ds\\
&&
+i\sum_{j=0}^{p-1}\int_0^t\hskip-0.25cma_j(s,x,u)D_x^jv(s)ds\\
=&&v+i\int_0^t\hskip-0.25cma_p(s)D_x^pv(s)ds+i\sum_{j=1}^{p-1}\int_0^t\hskip-0.25cma_j(s,x,u)D_x^jv(s)ds
\\
&&+i\int_0^t\hskip-0.2cm\Big(a_0(s,x,u)+\sum_{j=0}^{p-1}\partial_wa_j(s,x,u)D_x^ju
\Big)v(s)ds\\
=&&v-0+i\int_0^t\hskip-0.25cma_p(s)D_x^pv(s)ds+i\sum_{j=0}^{p-1}\int_0^t\tilde{a}_j(s,x,u)D_x^jv(s)ds-0\\
=:&&T_{0,u,0}(v),
\eeqsn
where
\beqsn
\tilde{a}_j(s,x,u)=\begin{cases}
a_j(s,x,u), &\  1\leq j\leq p-1\cr
\ds a_0(s,x,u)+\sum_{h=0}^{p-1}\partial_wa_h(s,x,u)D_x^hu, &\  j=0
\end{cases}
\eeqsn
and for every $u,u_0,f\in X$ the map $T_{u_0,u,f}:\ X\to X$ is defined by
\beqsn
T_{u_0,u,f}v:=&&v(t,x)-u_0(x)+i\int_0^ta_p(s)D_x^pv(s,x)ds\\
&&+i\sum_{j=0}^{p-1}\int_0^t\tilde{a}_j(s,x,u(s,x))D_x^jv(s,x)ds-i\int_0^t
f(s,x)ds.
\eeqsn

As in Remark \ref{remE}, we notice that $v$ is a solution of $T_{u_0,u,f}(v)\equiv0$ if and only if
it is a solution of the linearized Cauchy problem
\beqsn
\begin{cases}
  \tilde{P}_u(D)v(t,x)=f(t,x)\cr
  u(0,x)=u_0(x),
\end{cases}
\eeqsn
where $\tilde{P}_u(D)$ is obtained from $P_u(D)$ substituting $a_j$ with
$\tilde{a}_j$.

Therefore $v$ is a solution of $DT(u)v=h$ if and only if
$T_{0,u,0}(v)=h$;
writing
\beqsn
T_{0,u,0}(v)-h=T_{0,u,0}(v)-h_0-i\int_0^tD_th(s,x)ds=T_{h_0,u,D_t h}(v)
\eeqsn
with $h_0:=h(0,x)$, we have that $v$ is a solution of $DT(u)v=h$ if and only if it is a solution of $T_{h_0,u,D_t h}(v)=0$, i.e. it is a solution of the linearized Cauchy problem
\beqs
\label{5bis}
\begin{cases}
   \tilde{P}_u(D)v(t,x)=D_th(t,x)\cr
  v(0,x)=h_0(x).
\end{cases}
\eeqs

We now want to apply Theorem~\ref{iop} with $\tilde{P}_u(D)$ instead of
$P_u(D)$.
Note that conditions \eqref{ap}-\eqref{a1} are the same for $\tilde{P}_u(D)$
and $P_u(D)$, since there are no conditions for $j=0$.
Applying Theorem~\ref{iop}
we have that, for any $u,h\in X$ there is a
unique solution $v\in X$ of \eqref{5bis} satisfying the energy estimate
\beqsn
\|v(t,\cdot)\|_n^2\leq C_{n+\sigma,\gamma}e^{\big(1+\|u\|^{4p-3}_{4p-3}
\big)T}\left(\|h_0\|_{n+r}^2+\int_0^t
\|D_th(\tau,\cdot)\|_{n+r}^2d\tau\right)
\qquad\forall t\in[0,T],
\eeqsn
for any $r\geq\sigma$, which is exactly \eqref{6bis}. This completes the proof of the Lemma.
\end{proof}

We can therefore define the map
\beqs\label{essemaiusc}
S:\ X\times X&&\longrightarrow X\\
\nonumber
(u,h)&&\longmapsto v,
\eeqs
where $v$ is the unique solution of the Cauchy problem
\eqref{5bis}, i.e. of $DT(u)v=h$, and satisfies the
energy estimate \eqref{6bis}.

\begin{Lemma}
The map $S$ defined in \eqref{essemaiusc} is smooth tame.
\end{Lemma}
\begin{proof}
To prove that $S$ is smooth tame, we work by induction. The proof is divided into 4 steps. In steps $1$, $2$, $3$ we prove, respectively, that $S$, $DS$, $D^2S$ are tame maps; step 4 is the inductive step.

\smallskip

\noindent {\bf Step 1.} Let us show that $S$ is a tame map. To this aim we first remark that, for fixed $(u_0,h_0)\in X\times X$ and
$(u,h)$ in a neighborhood of $(u_0,h_0)$ we have that $C_n(u)$ is
bounded and hence, from the energy estimate \eqref{6bis},
\beqs
\label{46bis}
\|v(t,\cdot)\|_n^2\leq C'_n\3|h\3|_{n+r}^2
\qquad\forall t\in[0,T]
\eeqs
for some $C'_n>0$.
A similar estimate also holds for $D_tv$ since
\beqsn
\|D_tv(t,\cdot)\|_n=&&\Big\|\hskip-0.1cm-a_p(t)D_x^pv(t,\cdot)
-\hskip-0.1cm\sum_{j=0}^{p-1}\tilde{a}_j(t,\cdot,u)D_x^jv(t,\cdot)
\Big\|_n\hskip-0.2cm+\|D_th(t,\cdot)\|_n\\
\leq&& C(\|v(t,\cdot)\|_{n+p}+\3|h\3|_n)
\eeqsn
for some $C>0$.

Therefore
\beqs
\label{9bis}
\3|S(u,h)\3|_n=\sup_{t\in[0,T]}\left(\|v(t,\cdot)\|_n+\|D_tv(t,\cdot)\|_n
\right)
\leq C_n\3|h\3|_{n+r'}\leq C_n\3|(u,h)\3|_{n+r'}
\eeqs
for some $C_n>0$ and $r'\in \N, \ r'\geq\sigma+p$, and $S$ is tame.

\smallskip

\noindent{\bf Step 2.} We start by computing the Fr\'echet derivative of $S$,
for $(u,h),(u_1,h_1)\in X\times X$:
\beqs\nonumber
DS(u,h)(u_1,h_1)=&&\lim_{\varepsilon\to0}
\frac{S(u+\varepsilon u_1,h+\varepsilon h_1)-S(u,h)}{\varepsilon}\\\label{frech1}
=&&\lim_{\varepsilon\to0}\frac{v_\varepsilon-v}{\varepsilon}=\lim_{\varepsilon\to0}w_1^\varepsilon,
\eeqs
where $v$ is the solution of the Cauchy problem \eqref{5bis} and
$v_\varepsilon$ is the solution of the Cauchy problem
\beqsn
\begin{cases}
  \tilde{P}_{u+\varepsilon u_1}(D)v_\varepsilon=D_t(h+\varepsilon h_1)\cr
  v_\varepsilon(0,x)=h(0,x)+\varepsilon h_1(0,x).
\end{cases}
\eeqsn
Therefore
\beqs\label{venerdi}
\begin{cases}
  \tilde{P}_{u+\varepsilon u_1}(D)v_\varepsilon-\tilde{P}_u(D)v=\varepsilon D_t h_1\cr
  v_\varepsilon(0,x)-v(0,x)=\varepsilon h_1(0,x)
\end{cases}
\eeqs
and, writing explicitly the equation in \eqref{venerdi} we come to the equivalent equation
\beqsn
D_t(v_\varepsilon-v)+a_p(t)D_x^p(v_\varepsilon-v)
&+&\sum_{j=0}^{p-1}\tilde{a}_j(t,x,u+\varepsilon u_1)D_x^j
(v_\varepsilon-v)\\
&+&\sum_{j=0}^{p-1}\big(\tilde{a}_j(t,x,u+\varepsilon u_1)
-\tilde{a}_j(t,x,u)\big)D_x^jv
=\varepsilon D_th_1.
\eeqsn

This means that $w_1^\varepsilon$ in \eqref{frech1} satisfies
\beqsn
\begin{cases}
  \ds \tilde{P}_{u+\varepsilon u_1}w_1^\varepsilon=D_th_1
  -\sum_{j=0}^{p-1}\frac{\tilde{a}_j(t,x,u+\varepsilon u_1)-
\tilde{a}_j(t,x,u)}{\varepsilon}
  D_x^jv=:f_1^\varepsilon\cr
  w_1^\varepsilon(0,x)=h_1(0,x).
  \end{cases}
\eeqsn

If we prove that $\{w_1^\varepsilon\}_\varepsilon$ is a Cauchy sequence in $X$,
there
exists then $w_1\in X$ such that $w_1^\varepsilon\to w_1$ in $X$, so
that $DS(u,h)(u_1,h_1)=w_1$,
and $w_1$ solves the Cauchy problem
\beqs
\label{8bis}
\begin{cases}
  \tilde{P}_u(D)w_1=f_1\cr
  w_1(0,x)=h_1(0,x)
  \end{cases}
\eeqs
for
\beqs\label{A}
f_1:=\lim_{\varepsilon\to0}f_1^\varepsilon=D_th_1-\sum_{j=0}^{p-1}
\partial_w\tilde{a}_j(t,x,u)u_1D_x^jv.
\eeqs
Then, by Theorem~\ref{iop} the solution $w_1$ of the Cauchy problem
\eqref{8bis}
would satisfy the energy estimate
\beqsn
\|w_1(t,\cdot)\|_n^2\leq&&C_n(u)\Big(\|h_1(0,\cdot)\|_{n+r}^2+
\int_0^t\|f_1(\tau,\cdot)\|_{n+r}^2d\tau\Big)\\
\eeqsn
for $C_n(u)$ defined in \eqref{6bis}, and
\beqs
\nonumber
\|w_1(t,\cdot)\|_n\leq&&C'_n(u,u_1)\sup_{t\in[0,T]}\left(\|h_1(t,\cdot)\|_{n+r}+
\|D_th_1(t,\cdot)\|_{n+r}
+\|v(t,\cdot)\|_{n+r+p-1}\right)\\
\label{12bis}
\leq&& C''_n\left(\3|h_1\3|_{n+r'}+\3|h\3|_{n+r'}\right)
\eeqs
by \eqref{46bis},
for $(u,h)$ in a neighborhood of $(u_0,h_0)$ and
$(u_1,h_1)$ in a neighborhood of some fixed $(\bar{u}_1,\bar{h}_1)\in
X\times X$, and
for some $C'_n(u,u_1), C''_n>0$ and $r'\geq 2r+p-1$.
Also $$D_tw_1=-a_p(t)D_x^pw_1-\sum_{j=0}^{p-1}\tilde{a}_j(t,x,u)D_x^jw_1+f_1$$
would satisfy a similar estimate, so that the first derivative $DS$ would be
tame. Summing up, to gain that $DS$ is a tame map, it only remains to show that $\{w_1^\varepsilon\}_\varepsilon$
is a Cauchy sequence in $X$. To this aim, let us take $w_1^{\varepsilon_1}$ and $w_1^{\varepsilon_2}$ solutions, respectively, of the Cauchy problems
\beqs
\label{1011bis}
\begin{cases}
\tilde{P}_{u+\varepsilon_1u_1}(D)w_1^{\varepsilon_1}=f_1^{\varepsilon_1},\cr
w_1^{\varepsilon_1}(0,x)=h_1(0,x),
\end{cases}\qquad
\begin{cases}
\tilde{P}_{u+\varepsilon_2u_1}(D)w_1^{\varepsilon_2}=f_1^{\varepsilon_2,}\cr
w_1^{\varepsilon_2}(0,x)=h_1(0,x),
\end{cases}
\eeqs
then $w_1^{\varepsilon_1}-w_1^{\varepsilon_2}$ is solution of
\beqsn
&&D_tw_1^{\varepsilon_1}+a_p(t)D_x^pw_1^{\varepsilon_1}+
\sum_{j=0}^{p-1}\tilde{a}_j(t,x,u+\varepsilon_1u_1)D_x^jw_1^{\varepsilon_1}
-D_tw_1^{\varepsilon_2}-a_p(t)D_x^pw_1^{\varepsilon_2}
\\
&&-\sum_{j=0}^{p-1}
\tilde{a}_j(t,x,u+\varepsilon_2u_1)D_x^jw_1^{\varepsilon_2}
+\sum_{j=0}^{p-1}\tilde{a}_j(t,x,u+\varepsilon_1u_1)D_x^jw_1^{\varepsilon_2}
-\sum_{j=0}^{p-1}\tilde{a}_j(t,x,u+\varepsilon_1u_1)D_x^jw_1^{\varepsilon_2}
\\
&&=f_1^{\varepsilon_1}
-f_1^{\varepsilon_2}
\eeqsn
with initial condition $(w_1^{\varepsilon_1}-w_1^{\varepsilon_2})(0,x)=0$, i.e.
\beqsn
\begin{cases}
\tilde{P}_{u+\varepsilon_1u_1}(D)(w_1^{\varepsilon_1}-w_1^{\varepsilon_2})
=f_1^{\varepsilon_1}-f_1^{\varepsilon_2}
\ +\ds\sum_{j=0}^{p-1}\big(
\tilde{a}_j(t,x,u+\varepsilon_2u_1)-
\tilde{a}_j(t,x,u+\varepsilon_1u_1)\big)D_x^j
w_1^{\varepsilon_2}\cr
(w_1^{\varepsilon_1}-w_1^{\varepsilon_2})(0,x)=0.
\end{cases}
\eeqsn

By the energy estimate \eqref{Eoldv} and the Lagrange theorem, there exists
$u_{1,2}$ between $u+\varepsilon_1u_1$ and $u+\varepsilon_2u_1$ such that,
for all $t\in[0,T]$,
\beqsn
\|(w_1^{\varepsilon_1}-w_1^{\varepsilon_2})(t,\cdot)\|_n\leq C_n(u+\varepsilon_1u_1)
\bigg(&&\sup_{t\in[0,T]}\|f_1^{\varepsilon_1}(t,\cdot)
-f_1^{\varepsilon_2}(t,\cdot)\|_{n+r}\\
&&+\sum_{j=0}^{p-1}\sup_{t\in[0,T]}
\|\partial_wa_j(t,x,u_{1,2})(\varepsilon_1-\varepsilon_2)u_1D_x^j
w_1^{\varepsilon_2}\|_{n+r}\bigg)
\eeqsn
for some $C_n(u+\varepsilon_1u_1)>0$.
This goes to 0 as $\varepsilon_1\to\varepsilon_2\to0$ because
$f_1^{\varepsilon_1}
-f_1^{\varepsilon_2}\to 0$ and, being $H^s(\R)$ an algebra and satisfying
Sobolev inequality for $s>1/2$,
\beqsn
\|\partial_wa_j(t,x,u_{1,2})(\varepsilon_1-\varepsilon_2)u_1D_x^j
w_1^{\varepsilon_2}\|_{n+r}\leq
\|\partial_wa_j(t,x,u_{1,2})\|_{n+r}(\varepsilon_1-\varepsilon_2)
\|u_1\|_{n+r}\|w_1^{\varepsilon_2}\|_{n+r+j}
\eeqsn
is bounded for $(u,h)$ in a neighborhood of
$(u_0,h_0)$  and $(u_1,h_1)$ in a neighborhood of some fixed
$(\bar{u}_1,\bar{h}_1)\in X\times X$,
since $u_{1,2}$ is between $u+\varepsilon_1u_1$ and
$u+\varepsilon_2u_1$ and $\|w_1^{\varepsilon_2}\|_{n+r+j}$ is bounded by
the energy estimate
\beqsn
\|w_1^{\varepsilon_2}\|_{n+r+j}^2\leq C_{n+r+j}(u+\varepsilon_2 u_1)
\left(\|h_1(0,\cdot)\|_{n+2r+j}^2+\int_0^t
\|f_1^{\varepsilon_2}(\tau,\cdot)\|_{n+2r+j}^2d\tau\right).
\eeqsn

Then $\{w_1^{\varepsilon}\}_\varepsilon$ is a Cauchy sequence in $X$ and
the Fr\'echet derivative
\beqsn
DS:\ (X\times X)^2&&\longrightarrow X\\
\big((u,h),(u_1,h_1)\big)&&\longmapsto w_1,
\eeqsn
with $w_1$ solution of \eqref{8bis},
is tame by the above considerations.

\smallskip

\noindent{\bf Step 3.} Let us now consider the second derivative of $S$:
\beqsn
D^2S:\ (X\times X)^3&&\longrightarrow X\\
\big((u,h),(u_1,h_1),(u_2,h_2)\big)&&\longmapsto D^2S(u,h)(u_1,h_1)(u_2,h_2)
\eeqsn
defined by
\beqsn
D^2S(u,h)(u_1,h_1)(u_2,h_2)=&&
\lim_{\varepsilon\to0}\frac{DS(u+\varepsilon u_2,h+\varepsilon h_2)(u_1,h_1)-
DS(u,h)(u_1,h_1)}{\varepsilon}\\
=&&\lim_{\varepsilon\to0}\frac{\bar{w}_1^\varepsilon-w_1}{\varepsilon}=:\lim_{\varepsilon\to0}w_2^\varepsilon
\eeqsn
where $w_1$ is the solution of the Cauchy problem \eqref{8bis} and
$\bar{w}_1^\varepsilon$
is the solution of the Cauchy problem
\beqs\label{B}
\begin{cases}
  \ds \tilde{P}_{u+\varepsilon u_2}\bar{w}_1^\varepsilon=\bar{f}_1^\varepsilon:=D_th_1
  -\sum_{j=0}^{p-1}\partial_w
\tilde{a}_j(t,x,u+\varepsilon u_2)u_1
  D_x^jv\cr
  \bar{w}_1^\varepsilon(0,x)=h_1(0,x).
  \end{cases}
\eeqs
Writing
\beqsn
\bar{f}_1^\varepsilon-f_1=&&\tilde{P}_{u+\varepsilon u_2}(D)\bar{w}_1^\varepsilon
-\tilde{P}_u(D)w_1\\
=&&\tilde{P}_{u+\varepsilon u_2}(D)(\bar{w}_1^\varepsilon-w_1)+\sum_{j=0}^{p-1}
\big(\tilde{a}_j(t,x,u+\varepsilon u_2)-\tilde{a}_j(t,x,u)\big)D_x^jw_1
\eeqsn
we have from \eqref{A} and \eqref{B} that
\beqsn
&&\tilde{P}_{u+\varepsilon u_2}(D)w_2^\varepsilon
=\frac{f_1^\varepsilon-f_1}{\varepsilon}
-\sum_{j=0}^{p-1}\frac{\tilde{a}_j(t,x,u+\varepsilon u_2)-
\tilde{a}_j(t,x,u)}{\varepsilon}
D_x^jw_1\\
&&\hskip+0.5cm=-\sum_{j=0}^{p-1}
\frac{\tilde{a}_j(t,x,u+\varepsilon u_2)-\tilde{a}_j(t,x,u)}{\varepsilon}
D_x^jw_1
\ds-\sum_{j=0}^{p-1}\frac{\partial_w \tilde{a}_j(t,x,u+\varepsilon u_2)-
\partial_w\tilde{a}_j(t,x,u)}{\varepsilon}u_1D_x^jv
\\
&&\hskip+0.5cm=:f_2^\varepsilon
\eeqsn
i.e.
\beqsn
\begin{cases}
\ds \tilde{P}_{u+\varepsilon u_2}(D)w_2^\varepsilon=f_2^\varepsilon\cr
w_2^\varepsilon(0,x)=0.
\end{cases}
\eeqsn

Arguing as for $\{w_1^\varepsilon\}_\varepsilon$, we can prove that
$\{w_2^{\varepsilon}\}_\varepsilon$ is a Cauchy sequence and hence
$w_2^\varepsilon\to w_2\in X$, where $w_2$ is the solution of the Cauchy problem
\beqsn
\begin{cases}
\tilde{P}_u(D)w_2=-\ds\sum_{j=0}^{p-1}\partial_w \tilde{a}_j(t,x,u)u_2D_x^jw_1-
\ds\sum_{j=0}^{p-1}\partial^2_w\tilde{a}_j(t,x,u)u_1u_2D_x^jv=:f_2\cr
w_2(0,x)=0
\end{cases}
\eeqsn
and satisfies the following energy estimate
for $(u,h)$ in a neighborhood of $(u_0,h_0)$ and $(u_1,h_1)$, $(u_2,h_2)$
in a neighborhood of some fixed
$(\bar{u}_1,\bar{h}_1),(\bar{u}_2,\bar{h}_2)\in X\times X$:
\beqsn
\|w_2\|_n^2\leq&&C_n(u)\int_0^t\|f_2(\tau,\cdot)\|_{n+r}^2d\tau
\eeqsn
which gives
\beqsn
\|w_2\|_n\leq&&C'_n(u,u_1,u_2)
\sum_{j=0}^{p-1}\left(\|D_x^jw_1\|_{n+r}+\|D_x^jv\|_{n+r}\right)\\
\leq&&C'_n(u,u_1,u_2)\left(\|w_1\|_{n+r+p-1}+\|v\|_{n+r+p-1}\right)
\leq C''_n\left(\3|h_1\3|_{n+r''}+\3|h\3|_{n+r''}\right)
\eeqsn
for some $C'_n(u,u_1,u_2),C''_n>0$ and $r''\geq r+p-1+r'$, by \eqref{12bis} and
\eqref{46bis}. Then also $D^2S$ is tame.

\smallskip

\noindent{\bf Step 4.} We prove by induction on $m\in\N$
that, for all $m\geq2$,
\beqs
\label{15bis}
D^mS(u,h)(u_1,h_1)\cdots(u_m,h_m)=w_m
\eeqs
is the solution of the Cauchy problem
\beqs
\label{16bis}
\begin{cases}
\ds \tilde{P}_u(D)w_m=f_m\cr
w_m(0,x)=0,
\end{cases}
\eeqs
with
\beqs
\nonumber
f_m:=-\sum_{j=0}^{p-1}\partial_w\tilde{a}_j(t,x,u) u_mD_x^jw_{m-1}
\ds-\sum_{j=0}^{p-1}\partial_w^2\tilde{a}_j(t,x,u)u_{m-1}u_mD_x^jw_{m-2}
\\
\label{C''}
-\cdots\ds-\sum_{j=0}^{p-1}\partial_w^m\tilde{a}_j(t,x,u)u_1
\cdots u_{m-1}u_mD_x^jw_0
\eeqs
and $w_0:=v$, and satisfies, in a neighborhood of $(u,h)$, $(u_1,h_1)$,
\ldots $(u_m,h_m)$, the estimate
\beqs
\label{17bis}
\3|w_m\3|_n\leq C_n\sum_{j=0}^{m-1}\3|h_j\3|_{n+r(m)}
\eeqs
for some $C_n>0$ and $r(m)\in\N$, where $h_0:=h$.

Let us assume \eqref{15bis}-\eqref{17bis} to be true for all $j\leq m$ and let us prove them
for $j=m+1$:
\beqs
\nonumber
&&\quad D^{m+1}S(u,h)(u_1,h_1)\cdots(u_{m+1},h_{m+1})\\
\nonumber
=&&
\lim_{\varepsilon\to0}\frac{D^mS(u+\varepsilon u_{m+1}, h+\varepsilon u_{m+1})
(u_1,h_1)\cdots(u_m,h_m)-D^mS(u,h)(u_1,h_1)\cdots(u_m,h_m)}{\varepsilon}\\
\label{C'}
&&\qquad=\lim_{\varepsilon\to0}\frac{\bar{w}_m^\varepsilon-w_m}{\varepsilon}=:\lim_{\varepsilon\to0}w_{m+1}^\varepsilon,
\eeqs
where $w_m$ is the solution of \eqref{16bis} and $\bar{w}_m^\varepsilon$ is
the solution of
\beqsn
\begin{cases}
\ds \tilde{P}_{u+\varepsilon u_{m+1}}(D)\bar{w}_m^\varepsilon=f_m^\varepsilon\cr
\bar{w}_m^\varepsilon(0,x)=0
\end{cases}
\eeqsn
with
\beqs\nonumber
f_m^\varepsilon
:=-\sum_{j=0}^{p-1} \partial_w\tilde{a}_j(t,x,u+\varepsilon u_{m+1})u_mD_x^jw_{m-1}-\sum_{j=0}^{p-1}
\partial_w^2\tilde{a}_j(t,x,u+\varepsilon u_{m+1})u_{m-1}u_mD_x^jw_{m-2}
\\
\label{D}
-\cdots-\sum_{j=0}^{p-1}
\partial_w^m\tilde{a}_j(t,x,u+\varepsilon u_{m+1})u_1\cdots u_mD_x^jw_0.
\eeqs

Then
\beqs\nonumber
f_m^\varepsilon-f_m=&&\tilde{P}_{u+\varepsilon u_{m+1}}(D)\bar{w}_m^\varepsilon
-\tilde{P}_u(D)w_m\\
\nonumber
=&&D_t(\bar{w}_m^\varepsilon-w_m)+a_p(t)D_x^p(\bar{w}_m^\varepsilon-w_m)+\sum_{j=0}^{p-1}
\tilde{a}_j(t,x,u+\varepsilon u_{m+1})D_x^j(\bar{w}_m^\varepsilon
-w_m)\\
\nonumber
&&+\sum_{j=0}^{p-1}\big(\tilde{a}_j(t,x,u+\varepsilon u_{m+1})-
\tilde{a}_j(t,x,u)\big)D_x^jw_m\\
\label{C}
=&&\tilde{P}_{u+\varepsilon u_{m+1}}(D)(\bar{w}_m^\varepsilon-w_m)
+\sum_{j=0}^{p-1}\big(\tilde{a}_j(t,x,u+\varepsilon u_{m+1})-
\tilde{a}_j(t,x,u)\big)D_x^jw_m
\eeqs
and hence, by \eqref{C'}, \eqref{C}, \eqref{D} and \eqref{C''},
$w_{m+1}^\varepsilon$ is solution of the Cauchy problem
\beqsn
\begin{cases}
\tilde{P}_{u+\varepsilon u_{m+1}}(D)w_{m+1}^{\varepsilon}=f_{m+1}^{\varepsilon}\cr
w_{m+1}^{\varepsilon}(0)=0
\end{cases}
\eeqsn
where
\beqsn
f_{m+1}^\varepsilon:=&&-\sum_{j=0}^{p-1}
\frac{\tilde{a}_j(t,x,u+\varepsilon u_{m+1})-
\tilde{a}_j(t,x,u)}{\varepsilon}D_x^jw_m+\frac{f_m^\varepsilon-f_m}{\varepsilon}\\
=&&-\sum_{j=0}^{p-1}
\frac{\tilde{a}_j(t,x,u+\varepsilon u_{m+1})-
\tilde{a}_j(t,x,u)}{\varepsilon}D_x^jw_m\\
&&-\sum_{j=0}^{p-1}
\frac{\partial_w\tilde{a}_j(t,x,u+\varepsilon u_{m+1})-
\partial_w\tilde{a}_j(t,x,u)}{\varepsilon}u_mD_x^jw_{m-1}\\
&&-\sum_{j=0}^{p-1}
\frac{\partial^2_w\tilde{a}_j(t,x,u+\varepsilon u_{m+1})-
\partial^2_w\tilde{a}_j(t,x,u)}{\varepsilon}u_{m-1}u_mD_x^jw_{m-2}\\
&&\ \vdots\\
&&-\sum_{j=0}^{p-1}
\frac{\partial^m_w\tilde{a}_j(t,x,u+\varepsilon u_{m+1})-
\partial^m_w\tilde{a}_j(t,x,u)}{\varepsilon}u_1
\cdots u_mD_x^jw_0.
\eeqsn

Arguing as for $\{w_1^\varepsilon\}_\varepsilon$, we can prove that
$\{w_{m+1}^\varepsilon\}_\varepsilon$ is a Cauchy sequence and
therefore $w_{m+1}^\varepsilon\to w_{m+1}\in X$, where $w_{m+1}$ is the solution
of the Cauchy problem
\beqsn
\begin{cases}
\ds \tilde{P}_u(D)w_{m+1}=f_{m+1}\cr
w_{m+1}(0,x)=0,
\end{cases}
\eeqsn
with
\beqsn
f_{m+1}:&=&-\sum_{j=0}^{p-1}\partial_w \tilde{a}_j(t,x,u)u_{m+1}D_x^jw_m
-\sum_{j=0}^{p-1}\partial^2_w\tilde{a}_j(t,x,u)u_mu_{m+1}D_x^jw_{m-1}
\\
&&-\sum_{j=0}^{p-1}\partial^3_w\tilde{a}_j(t,x,u)
u_{m-1}u_mu_{m+1}D_x^jw_{m-2}
\cdots
-\sum_{j=0}^{p-1}\partial^{m+1}_w\tilde{a}_j(t,x,u)u_1\cdots
u_{m+1}D_x^jw_0
\eeqsn
and \eqref{16bis} is proved for $j=m+1$.
Moreover, by the energy estimate \eqref{Eoldv} and the inductive assumption
\eqref{17bis}
\beqsn
\|w_{m+1}\|_n^2\leq&&C_n(u)\int_0^t\|f_{m+1}(\tau,\cdot)\|_{n+r}^2d\tau
\eeqsn
and so, for
$(u,h)$ in a neighborhood of $(u_0,h_0)$ and $(u_1,h_1)$, \ldots, $(u_m,h_m)$
 in a neighborhood of some fixed
$(\bar{u}_1,\bar{h}_1),\ldots(\bar{u}_m,\bar{h}_m)\in X\times X$,
\beqsn
\|w_{m+1}\|_n\leq&&C'_n(u)\sum_{s=0}^m\sum_{j=0}^{p-1}
\|u_{m-s+1}\cdots u_{m}\cdot u_{m+1} D_x^jw_{m-s}\|_{n+r}\\
\leq&& C_n(u,u_1,\ldots,u_m)\sum_{j=0}^{p-1}\sum_{s=0}^m\|w_{m-s}\|_{n+r+j}\leq C'_n\sum_{j=0}^{p-1}\sum_{s=0}^m
\sum_{i=0}^{m-s-1}\3|h_i\3|_{n+r+j+r(m-s)}\\
\leq&&C''_n
\sum_{i=0}^{m-1}\3|h_i\3|_{n+r'(m)}
\eeqsn
for some $C'_n(u),C_n(u,u_1,\ldots,u_m),C'_n,
C''_n>0$, $r'(m)\in\N$.

Then also
\beqsn
\|D_tw_{m+1}\|_n=&&\|-a_p(t)D_x^pw_{m+1}-\sum_{j=0}^{p-1}\tilde{a}_j(t,x,u)
D_x^jw_{m+1}
+f_{m+1}\|_n\\
\leq&&C_n(u)\left(\|w_{m+1}\|_{n+p}+\|f_{m+1}\|_n\right)
\leq C'_n\sum_{i=0}^{m-1}\3|h_i\3|_{n+r''(m)}
\eeqsn
for some $C'_n>0$, $r''(m)=p+r'(m)$,
and for
$(u,h)$ in a neighborhood of $(u_0,h_0)$ and $(u_1,h_1)$, \ldots, $(u_m,h_m)$
 in a neighborhood of some fixed
$(\bar{u}_1,\bar{h}_1),\ldots(\bar{u}_m,\bar{h}_m)\in X\times X$.
Therefore \eqref{17bis} holds also for $m+1$.

We have thus proved \eqref{15bis}-\eqref{17bis}. In particular,
$D^mS$ is tame for every $m$ and hence $S$ is a smooth tame map.
\end{proof}
We are now ready to prove Theorem \ref{th1}.
\begin{proof}[Proof of Theorem \ref{th1}]
By Remark \ref{remE}, our goal is to look for a local solution $u$ of \eqref{2} as a
local solution of \eqref{4bis}, i.e. of
\beqs
\label{defu}
u(t,x)=&&u_0(x)-i\int_0^ta_p(s)D_x^pu(s,x)ds
-i\sum_{j=0}^{p-1}\int_0^ta_j(s,x,u(s,x))D_x^ju(s,x)ds\\
\nonumber
&&+i\int_0^tf(s,x)ds,
\eeqs
by definition \eqref{heart} of the map $T$.
To this aim, let us notice that from \eqref{defu} we have the Taylor expansion of the
first order of  $u$:
\beqs\nonumber
u(t,x)&=&u_0(x)-it\left(a_p(0)D_x^pu_0(x)+\sum_{j=0}^{p-1}
a_j(0,x,u_0(x))D_x^ju_0(x)
-f(0,x)\right)+o(t)
\\\label{defw}
&=:&w(t,x)+o(t),\qquad \mbox{as}\ t\to0.
\eeqs

The function $w\in X$ is in a neighborhood of the solution $u$
we are looking for,
if $t$ is sufficiently small. The idea of the proof is to
approximate $Tw$ by a function $\phi_\varepsilon$ identically zero for
$0\leq t\leq T_\varepsilon\leq T$ and apply
the Nash-Moser's Theorem~\ref{thNM}, in particular the fact that
$T$ is a bijection of a
neighborhood $U$ of $w$ onto a neighborhood $V$ of $Tw$.
If $\phi_\varepsilon\in V$, then by the local invertibility of $T$ there will be $u\in U$ such that
$Tu=\phi_\varepsilon\equiv0$ in $[0,T_\varepsilon]$ and hence the local
(in time) solution of \eqref{4bis} will be found.

To construct $\phi_\varepsilon$ we compute first (see the definition \eqref{heart} of $T$):
\beqsn
\partial_t (Tw(t,x))
=&&\partial_tw+ia_p(t)D_x^pw+i\sum_{j=0}^{p-1}a_j(t,x,w)D_x^jw-if(t,x),
\eeqsn
and using the definition \eqref{defw} of $w$ we get
\beqsn
\partial_t (Tw(t,x))=&&-ia_p(0)D_x^pu_0-i\sum_{j=0}^{p-1}a_j(0,x,u_0)D_x^ju_0+if(0,x)\\
&&+ia_p(t)D_x^pu_0+ta_p(t)D_x^p\Big(a_p(0)D_x^pu_0
+\sum_{j=0}^{p-1}a_j(0,x,u_0)D_x^ju_0-f(0,x)\Big)\\
&&+i\sum_{j=0}^{p-1}a_j(t,x,w)D_x^ju_0+t\sum_{j=0}^{p-1}a_j(t,x,w)D_x^j\Big(a_p(0)D_x^pu_0
+\sum_{j=0}^{p-1}a_j(0,x,u_0)D_x^ju_0\\
&&-f(0,x)\Big)-if(t,x)\\
=&&i[a_p(t)-a_p(0)]D_x^pu_0
+i\sum_{j=0}^{p-1}\big[a_j(t,x,w)-a_j(0,x,u_0)\big]D_x^ju_0\\
&&+a_p(t)tD_x^p\bigg[a_p(0)D_x^pu_0+\sum_{j=0}^{p-1}
a_j(0,x,u_0)D_x^ju_0-f(0,x)\bigg]\\
&&+\sum_{j=0}^{p-1}a_j(t,x,w)tD_x^j\bigg[a_p(0)D_x^pu_0+\sum_{s=0}^{p-1}
a_s(0,x,u_0)D_x^su_0-f(0,x)\bigg]\\
&&+i\big(f(0,x)-f(t,x)\big).
\eeqsn
Therefore
\beqsn
\|\partial_tTw(t,\cdot)\|_n\leq&&\sup_{t\in[0,T]}|a_p(t)-a_p(0)|
\cdot\|u_0\|_{n+p}+\sum_{j=0}^{p-1}\left\|\left[a_j(t,x,w)-a_j(0,x,u_0)\right]D_x^ju_0
\right\|_n\\
&&+t\sup_{t\in[0,T]}|a_p(t)|\cdot\left\|a_p(0)D_x^pu_0
+\sum_{s=0}^{p-1}a_s(0,x,u_0)D_x^su_0-f(0,x)\right\|_{n+p}\\
&&+t\sum_{j=0}^{p-1}\left\|a_j(t,x,w)D_x^j\left[a_p(0)D_x^pu_0+
\sum_{s=0}^{p-1}a_s(0,x,u_0)D_x^su_0-f(0,x)\right]\right\|_n\\
&&+\|f(0,x)-f(t,x)\|_n.
\eeqsn
By Lagrange theorem and for $t$ sufficiently small (so that $w$ is
in a sufficiently small neighborhood of $u_0$):
\beqs
\nonumber
\|\partial_tTw(t,\cdot)\|_n\leq&&C_0(a_p,u_0)t+\sum_{j=0}^{p-1}
C(a_j,u_0)t+C_1(a_p,\ldots,a_0,u_0,f)t+C_1(f)t\\
\label{19bis}
\leq&&C(a_p,\ldots,a_0,u_0,f)t
\eeqs
for some positive constants $C_0(a_p,u_0)$, $C(a_j,u_0)$, $C_1(a_p,\ldots,a_0,u_0,f)$, $C_1(f)$,
$C(a_p,\ldots,a_0,u_0,f)$ depending only on the variables specified there.

Let us now choose $\rho\in C^\infty(\R)$ with $0\leq\rho\leq1$ and
\beqsn
\rho(s)=\begin{cases}
0, &s\leq1\cr
1, &s\geq2.
\end{cases}
\eeqsn
Define then
\beqsn
\phi_\varepsilon(t,x):=
\int_0^t\rho\left(\frac s\varepsilon\right)(\partial_tTw)(s,x)ds,
\eeqsn
and note that $\phi_\varepsilon\equiv0$ for $0\leq t\leq\varepsilon$. We are going to prove
that, for every fixed neighborhood $V$ of $Tw$ in the topology of
$X=C^1([0,T];H^\infty(\R))$, we have $\phi_\varepsilon\in V$ if
$\varepsilon$ is sufficiently small.
Indeed, by definition of $\phi_\varepsilon$ and using \eqref{19bis},
\beqs
\nonumber
\|Tw-\phi_\varepsilon\|_n=&&\left\|\int_0^t
\left(1-\rho\left(\frac s\varepsilon\right)\right)
(\partial_tTw)(s,\cdot)ds\right\|_n\leq
\int_0^{2\varepsilon}\left\|\left(1-\rho\left(\frac s\varepsilon\right)\right)
(\partial_tTw)(s,\cdot)\right\|_n ds\\
\label{20bis}
\leq&&C(a_p,\ldots,a_0,u_0,f)
\int_0^{2\varepsilon}\!s\,ds=C(a_p,\ldots,a_0,u_0,f)2\varepsilon^2.
\eeqs
Moreover
\beqs
\nonumber
\|\partial_t(Tw-\phi_\varepsilon)\|_n=&&\left\|\partial_tTw(t,\cdot)-
\rho\left(\frac t\varepsilon\right)(\partial_tTw)(t,\cdot)\right\|_n\leq\left(1-\rho\left(\frac t\varepsilon\right)\right)
\|\partial_tTw(t,\cdot)\|_n\\
\label{21bis}
\leq&&
\left(1-\rho\left(\frac t\varepsilon\right)\right)
C(a_p,\ldots,a_0,u_0,f)t\leq2C(a_p,\ldots,a_0,u_0,f)\varepsilon,
\eeqs
again by \eqref{19bis} and looking at the support of $1-\rho\left(t/\varepsilon\right)$.

From \eqref{20bis} and \eqref{21bis} we thus have, for
$0<\varepsilon<1$, that
\beqsn
\3|Tw-\phi_\varepsilon\3|_n\leq 2C(a_p,\ldots,a_0,u_0,f)\varepsilon
\eeqsn
and hence $\phi_\varepsilon\in V$ for $\varepsilon$ sufficiently small,
where $V$ is the neighborhood of $Tw$ such that $T:\ U\to V$ is
invertible.

Then, there exists $u\in U\subset X$ such that $Tu=\phi_\varepsilon$ and
hence, in particular,
\beqsn
Tu\equiv0\qquad\mbox{for}\ 0\leq t\leq\varepsilon.
\eeqsn

This proves that $u\in C^1([0,\varepsilon];H^\infty(\R))$ is a
local solution of the Cauchy problem \eqref{2}.

Uniqueness follows by standard arguments.
As a matter of fact, if $u,v$ are two solutions of the Cauchy problem
\eqref{2}, we have
\beqsn
0=P_u(D)u-P_v(D)v
=&&P_u(D)(u-v)+\sum_{j=0}^{p-1}\big(a_j(t,x,u)-a_j(t,x,v)\big)D_x^jv\\
=&&P_u(D)(u-v)+\sum_{j=0}^{p-1}\int_v^u\partial_wa_j(t,x,s)ds\,D_x^jv\\
=&&P_u(D)(u-v)+\sum_{j=0}^{p-1}\int_0^1\partial_wa_j(t,x,v+t(u-v))(u-v)dt
\,D_x^jv\\
=&&\left( P_u(D)+\sum_{j=0}^{p-1}\int_0^1\partial_wa_j(t,x,v+t(u-v))dt
\,D_x^jv\right)(u-v)\\
=:&&\tilde{\tilde{P}}(u-v).
\eeqsn
Therefore, for fixed $u,v\in X$, the function $w:=u-v$ solves the linear
Cauchy problem
\beqs
\label{CPC1}
\begin{cases}
\tilde{\tilde P}w=0\cr
w(0,x)=0,
\end{cases}
\eeqs
and since $\tilde{\tilde P}$ is of the same form as $P_u(D)$ with
\beqsn
\tilde{\tilde a}_0(t,x,u):=a_0(t,x,u)+
\sum_{j=0}^{p-1}\int_0^1\partial_wa_j(t,x,v+t(u-v))dt
\,D_x^jv
\eeqsn
instead of $a_0(t,x,u)$, and has therefore the same kind of regularity on
the coefficients. By the uniqueness of the linearized Cauchy problem (given by
Theorem~\ref{thabz}), we finally have that $w=0$.
Therefore $u=v$ and uniqueness is proved.
\end{proof}

\section{Further generalizations}
\label{gen}
In this section we focus on generalizations of Theorem \ref{th1};
we first consider the Cauchy problem \eqref{2} with $a_p=a_p(t,x)$, $x\in\R$,
and then give an idea on how the result can be extended to the case $x\in\R^n$.

The dependence of $a_p$ on $x$ means
that in the explicit expression of the symbol $\sigma(e^{-\Lambda }A e^\Lambda)$
in \eqref{33} some new terms containing $D_x^\beta a_p(t,x)$ appear for
$\beta\neq 0$.

By assuming
$a_p\in C([0,T],\mathcal B^\infty(\R))$ with $a_p(t,x)\in\R$ and
 \beqs
  \label{1101}
  &&\quad \ |\Im (D_x^\beta a_p)(t,x)|\leq
\frac{C}{\langle x\rangle^{\frac{p-[\beta/2]}{p-1}}}\,,\quad 0\leq
\left[\frac\beta 2\right]\leq p-1,\ \beta\neq0
\\
\label{1102}
&&\quad \ |\Re (D_x^\beta a_p)(t,x)|\leq C, \qquad 0\leq
\beta\leq p-1,
\eeqs
in analogy with  \eqref{im}-\eqref{re},
we shall retrace here below the proof of Theorem \ref{th1}. 

Notice that $a_p\in C([0,T];{\mathcal B}^\infty(\R))$ implies that
condition \eqref{1102} is authomatically satisfied, while $a_p$ real valued
implies
that condition \eqref{1101} reduces to 
\beqsn
\quad\ |D_x^\beta a_p(t,x)|\leq
\frac{C}{\langle x\rangle^{\frac{p-[\beta/2]}{p-1}}}\,,
\qquad \ 0\leq\left[\frac\beta 2\right]\leq p-1,\ \beta\ \mbox{odd}.
\eeqsn
Therefore we can prove the following theorem:

\begin{Th}
  \label{thxxx}
  Let $p\geq2$ and consider the following $p$-evolution operator:
  \beqs
  \label{Px}
P_u(D)u:=&&D_tu +a_p(t,x)D_x^pu+\ds\sum_{j=0}^{p-1}a_j(t,x,u)
D_x^ju,
  \eeqs
  where $a_p\in C([0,T];{\mathcal B}^\infty(\R))$ with $a_p(t,x)\in \R$,
  and $a_j\in C([0,T];C^\infty(\R\times \C))$ with
  $x\mapsto a_j(t,x,w)\in{\mathcal B}^\infty(\R)$, for
  $0\leq j\leq p-1$.

  Let us assume that there exist constants $C_p>0$ and $C>0$ and a function
$\gamma:\ \C\to\R^+$ of class $C^7$
such that, for all $(t,x,w)\in[0,T]\times\R\times\C$:
\beqs
\label{apcp}
&&\quad \ a_p(t,x)\geq C_p,
\\
\nonumber
&&\quad\ |D_x^\beta a_p(t,x)|\leq
\frac{C}{\langle x\rangle^{\frac{p-[\beta/2]}{p-1}}},
\qquad  0\leq
\left[\frac\beta 2\right]\leq p-1,\ \beta\ \mbox{odd},
\\
\nonumber
&&\quad \ |\Im (D_x^\beta a_j)(t,x,w)|\leq
\frac{C \gamma(w)}{\langle x\rangle^{\frac{j-[\beta/2]}{p-1}}},\quad 0\leq
\left[\frac\beta 2\right]\leq j-1,\ 3\leq j\leq p-1,
\\
\nonumber
&&\quad \ |\Re (D_x^\beta a_j)(t,x,w)|\leq C \gamma(w)\qquad 0\leq \beta\leq j-1,\
3\leq j\leq p-1,
\\
\nonumber
&&\quad \ | (D_w^\gamma D_x^\beta a_j)(t,x,w)|\leq
\frac{C \gamma(w)}{\langle x\rangle^{\frac{j-[(\gamma+\beta)/2]}{p-1}}},\quad
\gamma\geq 1, \beta\geq 0, \left[\frac{\gamma+\beta} 2\right]\leq j-1,\
3\leq j\leq p-1,
\\
\nonumber
&&\quad \ |\Re a_2(t,x,w)|\leq C\gamma(w),
\\
\nonumber
&&\quad \ |\Im a_2(t,x,w)|\leq \frac{C\gamma(w)}{\langle x\rangle^{\frac{2}{p-1}}},
\\
\nonumber
&&\quad \ |\Im a_1(t,x,w)|+|\Im D_x a_2(t,x,w)|+|D_w a_2(t,x,w)|\leq
\frac{C\gamma(w)}{\langle x\rangle^{\frac{1}{p-1}}}.
\eeqs

Then the Cauchy problem \eqref{2}, for $P_u(D)$ defined as in \eqref{Px},
is locally in time well-posed in $H^\infty$:
for all $f\in C([0,T];H^\infty(\R))$ and $u_0\in H^\infty(\R)$, there
exists $0<T^*\leq T$ and a unique solution $u\in C([0,T^*];
H^\infty(\R))$ of \eqref{2}.
\end{Th}

\begin{proof}
  We remark that from the assumptions \eqref{1101}-\eqref{1102} we can obtain
  \eqref{euna}-\eqref{edue} also for $j=p$ (indeed, inequalities
\eqref{VAR2}-\eqref{VAR3} are valid also for $j=p$ with a fixed
constant $\gamma(u)=\gamma$).

Now we follow the proof of Theorem~\ref{iop} (see also \cite{ABZ})
outlining the needed changes.
In formula \eqref{33} of step 1, on the one hand the symbol $A_I$ remains
unvaried if $a_p=a_p(t,x)$, so \eqref{38} is unvaried too; on the other hand, 
$A_{II}$ has terms $D_x^\beta a_p(t,x)$ which are now different from zero also
for $\beta\neq0$.
Deriving \eqref{34} from \eqref{AC1} and \eqref{AC2} we thus have to take
into account these terms.
The estimates of the order and decay of the terms in \eqref{AC1} are the same
as in the case $a_p=a_p(t)$, while in the estimate of the order and decay
of the terms in \eqref{AC2}, the only term which works differently from the
case $a_p=a_p(t)$ is $-(\partial_\xi\Lambda)(D_xia_p)\xi^p$.
This is sum of terms of the form
$-(\partial_\xi\lambda_{p-k})(\partial_x a_p)\xi^p$: for $k=1$ the term
$-(\partial_\xi\lambda_{p-1})(\partial_x a_p)\xi^p$ is of order zero because
$\partial_\xi\lambda_{p-1}$ has support in the set
$\{\frac12\langle\xi\rangle_h^{p-1}\leq\langle x\rangle\leq
\langle\xi\rangle_h^{p-1}\}$;
the other terms
\beqs
\label{PK4}
-(\partial_\xi\lambda_{p-k})(\partial_x a_p)\xi^p,
\qquad\  2\leq k\leq p-1,
\eeqs
have order $p-k$ and not $p-k-1$, as it was in \eqref{PK1}.

Therefore all terms of $\Re \left.A_{II}\right|_{\ord(p-k)}$ satisfy
\eqref{34} except for the terms in \eqref{PK4}, that we shall treat
separately here below, following the same ideas as in the proof of the
ivertibility of $e^\Lambda$ (cf. \cite{ABZ}).

By \eqref{PK2} and \eqref{1101}:
\beqs
\nonumber
|\partial_\xi\lambda_{p-k}\partial_x a_p\xi^p|\leq &&
C_{1,0}M_{p-k}\langle x\rangle^{\frac{k-1}{p-1}}\langle\xi\rangle_h^{-k}
\chi_E(x)\frac{C}{\langle x\rangle^{\frac{p}{p-1}}}\langle\xi\rangle_h^p\\
\nonumber
=&&C_{1,0}C\,\frac{M_{p-k}}{\langle\xi\rangle_h}\,
\frac{\langle\xi\rangle_h^{p-k+1}}{\langle x\rangle^{\frac{p-k+1}{p-1}}}
\chi_E(x)\\
\nonumber
\leq&&C_{1,0}C\,\frac{M_{p-k}}{h}\,
\frac{\langle\xi\rangle_h^{p-k+1}}{\langle x\rangle^{\frac{p-k+1}{p-1}}}\\
\label{PK3}
\leq&&C_{1,0}C\,\frac{\langle\xi\rangle_h^{p-k+1}}{\langle x\rangle^{\frac{p-k+1}{p-1}}}
\eeqs
if $h\geq M_{p-k}$.

This means that we can insert $(\partial_\xi\lambda_{p-k})(\partial_x a_p)\xi^p$
at level $p-k+1$, since it has the ``right decay'' for the level
$p-k+1$ and satisfies an estimate of the form \eqref{stp3}, with a constant
$C_{1,0}C$ that does not depend on any of the $M_{p-k}$, for $k\geq2$.

Therefore we shall insert $(\partial_\xi\lambda_{p-k})(\partial_x a_p)\xi^p$ in
$\Re \left.A_{II}\right|_{\ord(p-k+1)}$ instead of
$\Re \left.A_{II}\right|_{\ord(p-k)}$, for $2\leq k\leq p-1$, and act as if
\eqref{34} holds as it was in Theorem~\ref{iop}.

All the other steps are based on the estimates \eqref{38} and \eqref{34},
so that their proof follows 
as in Theorem \ref{iop}, thanks to the added assumptions
\eqref{1101}, \eqref{1102}, which ensure that the new terms still have
the "right decay for the right level" and depend on the ``right
constants'' $M_{p-k}$, if we choose $h\geq\max\{M_{p-2},M_{p-3},
\ldots,M_1\}$ and large enough to ensure the invertibility of the operator
$e^\Lambda$ as in \eqref{sole}.

Finally the Nash-Moser scheme of Section \ref{sec4} does not involve the
dependence of $a_p$ on $x$.

Therefore Theorem \ref{thxxx} is proved.
\end{proof}

We conclude this paper with the following remark about the generalization
of Theorems \ref{th1} and \ref{thxxx} to the case $x\in\R^n$, $n\geq 2$.

\begin{Rem}\label{Rn}
  \begin{em}
    In the proof of Theorems \ref{th1} and \ref{thxxx}, the
    symbol $\Lambda=\lambda_1+\cdots+\lambda_{p-1}$ was constructed in
    \eqref{26},
  following \cite{KB},  in a way such that in Steps 2 and 3 of the
  proof of Theorem \ref{iop}, in order to apply the sharp-G{\aa}rding
  theorem, we got
    \beqs
  \label{?}
  \quad \Re (ipa_pD_x\lambda_{p-k}\xi^{p-1}+i\tilde a_{p-k})= pa_p
  \partial_x\lambda_{p-k}\xi^{p-1}-\Im \tilde a_{p-k}\geq -C(u), \quad k=1,
  \ldots,p-1,
  \eeqs
  where $\tilde a_{p-k}$ was given by the sum of $a_{p-k}\xi^{p-k}$ and (possible)
  other symbols of order $p-k$ with the "right decay for the right level"
  and dependence on the ``right constants''.
  
  For the case of more space variables we have to choose
  $\Lambda=\lambda_1+\cdots+\lambda_{p-1}$ in order that it satisfies a
  pseudo-differential inequality of the form:
\beqs
  \label{466}
  \sum_{j=1}^npa_p(t,x)\partial_{x_j}\lambda_{p-k}|\xi|^{p-2}\xi_j-\Im
  \tilde{\tilde{a}}_{p-k}
  \geq -C(u), \quad k=1,
  \ldots,p-1,
  \eeqs  
  where $\tilde{\tilde{a}}_{p-k}$ is a symbol of order $p-k$ with the
  ``right decay'' and depending on the ``right constants'', i.e.
  \beqs
  \label{448}
  |\Im \tilde{\tilde{a}}_{p-k}|\leq C(M_{p-1},\ldots,M_{p-k+1})
  \frac{\langle\xi\rangle_h^{p-k}}{\langle
    x\rangle^{\frac{p-k}{p-1}}}
    \eeqs
    for some $C(M_{p-1},\ldots,M_{p-k+1})>0$.

    A solution $\lambda_{p-k}$ to \eqref{466} can be constructed,
    following the ideas of \cite{KB}, by solving the equation
  \beqs
  \label{467}
  \sum_{j=1}^npC_p\partial_{x_j}\lambda_{p-k}|\xi|^{p-2}\xi_j=|\xi|g_k(x,\xi)
  \eeqs
  for $C_p$ as in \eqref{apcp}, and
  for some positive function $g_k(x,\xi)$ with a decay as in \eqref{448} and
  large enough so that \eqref{466} will therefore be satisfied.

  But a solution, for large $|\xi|$, of an equation of the form
  \beqsn
  \sum_{j=1}^n\xi_j\partial_{x_j}\lambda(x,\xi)
  =|\xi|g(x,\xi)
  \eeqsn
is given by (cf. \cite{KB})
\beqs
\label{deflambda}
  \lambda(x,\xi)=
  \int_0^{x\cdot\xi/|\xi|}g\left(x-\tau\frac{\xi}{|\xi|},\xi\right)d\tau,
  \eeqs
  so that the functions $\lambda_{p-k}$ can be constructed by means of
  functions  $\lambda$ for the form \eqref{deflambda},
  as explained in \cite{CR2}.
\end{em}
\end{Rem}


\section{Appendices}

\begin{appendices}
\section{The Nash-Moser Theorem} \label{app:a}

We recall here the basic notion of the Nash-Moser theory as in \cite{DA},
\cite{H}.

\begin{Def}
A {\em graded} Fr\'echet space $X$ is a Fr\'echet space whose topology is
generated by a {\em grading}, i.e. an increasing sequence of semi-norms:
\beqsn
\|x\|_n\leq\|x\|_{n+1},\qquad \forall n\in\N_0, \, x\in X.
\eeqsn
\end{Def}

\begin{Def}
For two graded Fr\'echet spaces $X,Y$, a linear map $L:\ X\to Y$ is said to be
a {\em tame linear map} if there exist $r,n_0\in\N$ such that for every
integer $n\geq n_0$ there is a constant $C_n>0$, depending only on $n$, s.t.
\beqs
\label{2.2DA}
\|Lx\|_n\leq C_n\|x\|_{n+r}\qquad\forall x\in X.
\eeqs
The number $n_0$ is called the {\em base} and $r$ the {\em degree} of the
{\em tame estimate} \eqref{2.2DA}.
\end{Def}
\begin{Def}
  Given a Banach space $B$, the {\em space of exponentially decreasing
    sequences} $\Sigma(B)$ is the graded space of all sequences
  $\{v_k\}_{k\in\N_0}\subset B$ such that
  \beqsn
  \|\{v_k\}\|_n:=\sum_{k=0}^{+\infty}e^{nk}\|v_k\|_B<+\infty\qquad
  \forall n\in\N_0.
  \eeqsn
  \end{Def}
\begin{Def}
  A graded space $X$ is said to be {\em tame} if there exist a Banach space $B$
  and two tame linear maps $L_1:\, X\to\Sigma(B)$ and $L_2:\,\Sigma(B)\to X$
  such that $L_2\circ L_1$ is the identity on $X$.
  \end{Def}

\begin{Rem}
  \begin{em}
    The property of being tame is stable under the usual operations like
    direct sum, product, etc. (cf. \cite{DA}, \cite{H}).
      \end{em}
\end{Rem}

\begin{Ex}
  \label{proptame}
  The space $C^1([0,T]; H^\infty(\R^m))$, $m\geq 1$, endowed with the family of seminorms
  \beqsn
  \3|g\3|_n:=\sup_{[0,T]}\left(\|g(t,\cdot)\|_n+
  \|D_tg(t,\cdot)\|_n\right),\qquad n\in\N_0,
  \eeqsn
  is a tame space.
\end{Ex}

\begin{proof}
Since $H^\infty(\R^m)$ is tame (see \cite{CNT}),
there exist a Banach space $B$ and two tame linear maps
\beqsn
&&L_1:\ H^\infty(\R^m)\longrightarrow \Sigma(B)\\
&&\hspace*{23mm} g\longmapsto\{g_k\}\\
&&L_2:\ \Sigma(B)\longrightarrow H^\infty(\R^m)\\
&&\hspace*{11mm} \{g_k\}\longmapsto g
\eeqsn
such that
\beqs
\label{C1}
&&\|L_1(g)\|_n=\|\{g_k\}\|_n:=\sum_{k=0}^{+\infty}e^{nk}\|g_k\|_B
\leq C_n\|g\|_{n+r}\qquad n\geq n_0\\
\nonumber
&&\|L_2(\{g_k\})\|_n=\|g\|_n\leq C'_n\|\{g_k\}\|_{n+r'}\qquad n\geq n_0
\eeqs
for some $C_n,C'_n,>0$, $n_0,r,r'\in\N$, and
\beqs
\label{C3}
L_2\circ L_1(g)=g\qquad\forall g\in H^\infty(\R^m).
\eeqs

We construct the linear map
\beqsn
\tilde{L}_1:\ C^1([0,T];H^\infty(\R^m))&&\longrightarrow
C^1([0,T];\Sigma(B))\\
g(t,x)&&\longmapsto \tilde{L}_1g(t,x)
\eeqsn
defined, for every fixed $t\in[0,T]$, by
$\tilde{L}_1g=L_1g(t,\cdot)=
\{g_k(t,\cdot)\}$.
Clearly $\{g_k(t,\cdot)\}\in\Sigma(B)$ for all $t\in[0,T]$ by construction.
We now prove that $\tilde{L}_1$ is well defined. Let us first remark
that, since $g\in C^1([0,T];H^\infty(\R^m))$, there exists
$g'\in C^0([0,T];H^\infty(\R^m))$ such that
\beqs
\label{C2}
\lim_{h\to0}\left\|\frac{g(t+h,\cdot)-g(t,\cdot)}{h}-g'(t,\cdot)\right\|_n=0
\qquad\forall n\in\N_0.
\eeqs

In order to show that $\tilde{L}_1g=\{g_k(t,\cdot)\}\in C^1([0,T];\Sigma(B))$
we shall prove that there exists, in $\Sigma(B)$, the $\partial_t$
derivative of $\tilde{L}_1g=\{g_k(t,\cdot)\}$ and this is given by
$L_1(g')=\{(g')_k(t,\cdot)\}$.
Indeed, for all $n\in\N$, from \eqref{C1} and \eqref{C2} we have that
\beqsn
&&\sum_{k=0}^{+\infty}e^{nk}\left\|\frac{g_k(t+h,\cdot)-g_k(t,\cdot)}{h}-
(g')_k(t,\cdot)\right\|_B\\
=&&
\left\|L_1\left(\frac{g(t+h,\cdot)-g(t,\cdot)}{h}-
g'(t,\cdot)
\right)\right\|_n\\
\leq&&C_n\left\|\frac{g(t+h,\cdot)-g(t,\cdot)}{h}-
g'(t,\cdot)\right\|_{n+r}\longrightarrow 0\quad\mbox{as}\ h\to0.
\eeqsn

Therefore $(g_k)'(t,\cdot)=(g')_k(t,\cdot)$ and
the linear map $\tilde{L}_1$
is a linear tame map because of \eqref{C1}:
\beqsn
\3|\tilde{L}_1(g)\3|_n:=&&\sup_{t\in[0,T]}\left(\sum_{k=0}^{+\infty}
e^{nk}\|g_k(t,\cdot)\|_n+\sum_{k=0}^{+\infty}
e^{nk}\|(g_k)'(t,\cdot)\|_n\right)\\
\leq&& C_n\sup_{t\in[0,T]}\left(\|g(t,\cdot)\|_{n+r}+\|g'(t,\cdot)\|_{n+r}
\right)\\
=&&C_n\3|g\3|_{n+r}.
\eeqsn
Analogously we can construct a tame linear map
\beqsn
\tilde{L}_2:\ C^1([0,T];\Sigma(B))&&\longrightarrow
C^1([0,T];H^\infty(\R^m))\\
\{g_k(t,x)\}&&\longmapsto g(t,x)
\eeqsn
defined by $\tilde{L}_2(\{g_k\})(t,\cdot)=L_2(\{g_k(t,\cdot)\})=g(t,\cdot)$ for
all $t\in[0,T]$.

Moreover, $\tilde{L}_2\circ \tilde{L}_1$ is the identity map
on $C^1([0,T];H^\infty(\R^m))$ by construction.

The proof is complete.
\end{proof}

For nonlinear maps, the definition of tame map is given by:
\begin{Def}
  Let $X,Y$ be two graded spaces, $U\subset X$ and $T:\ U\to Y$.
  We say that $T$ satisfies a {\em tame estimate} of degree $r$ and base $n_0$
  if for every integer $n\geq n_0$ there exists a constant $C_n>0$ such that
  \beqs
  \label{tameestimate}
  \|T(u)\|_n\leq C_n(1+\|u\|_{n+r})\qquad\forall u\in U.
  \eeqs
  A map $T$ defined on an open set $U$ is said to be {\em tame} if
  it satisfies a tame estimate \eqref{tameestimate} in a neighbourhood of
  each point $u\in U$ (with constants $r,n_0$ and $C_n$ which may depend on the
  neighbourhood).
\end{Def}

\begin{Rem}
  \begin{em}
    Let us remark that a linear map is tame if and only if it is a
    tame linear map. Moreover, the composition of tame
    maps is tame (cf. \cite{DA}, \cite{H}).
      \end{em}
\end{Rem}

Recalling the notion of {\em Fr\'echet} derivative $DT(u)v$ of a map
$T:\ U\subset X\to Y$ at $u\in U$ in the direction $v\in X$, as
\beqs
\label{2.6DA}
DT(u)v:=\lim_{\epsilon\to0}\frac{T(u+\epsilon v)-T(u)}{\epsilon},
\eeqs
we say that $T$ is $C^1(U)$ if it is differentiable, in the sense that the
limit \eqref{2.6DA} exists, and if the derivative
$DT:\ U\times X\to Y$ is continuous.

Recursively, we can define the successive derivatives
$D^nT:\ U\times X^n\to Y$ and say that $T$ is $C^\infty(U)$ if all the Fr\'echet
derivatives of $T$ exist and are continuous.

\begin{Def}
  Given two graded spaces $X,Y$ and an open subset $U$ of $X$, we say that a
  map $T:\ U\to Y$ is {\em smooth tame} if it is $C^\infty$ and $D^nT$ is
  tame for all $n\in\N_0$.
\end{Def}

\begin{Rem}
  \label{remA1}
  \begin{em}
    Sums and compositions of smooth tame maps are smooth
    tame. Moreover, linear and nonlinear partial differential operators and
    integration are smooth tame (cf. \cite{DA}, \cite{H}).
      \end{em}
\end{Rem}

We finally recall the Theorem of Nash-Moser (see \cite{H}):
\begin{Th}[Nash-Moser-Hamilton]
  \label{thNM}
  Let $X,Y$ be tame spaces, $U$ an open subset of $X$ and $T:\ U\to Y$ a
  smooth tame map. Assume that the equation $DT(u)v=h$ has a unique solution
  $v:=S(u,h)$ for all $u\in U$ and $h\in Y$ and assume that $S:\ U\times Y\to
  X$ is smooth tame. Then $T$ is locally invertible and each local inverse
  is smooth tame.
\end{Th}

\section{Sharp-G{\aa}rding and
Fefferman-Phong ine\-qua\-li\-ties} \label{app:b}

Let $A(x,D_x)$ be a pseudo-differential operator of order $m$ on $\R$
with symbol $A(x,\xi)$ in the standard class $S^m$ defined by
\beqs
\label{calphabeta}
|\partial_\xi^\alpha D_x^\beta A(x,\xi)|\leq C_{\alpha,\beta}
\langle\xi\rangle^{m-\alpha}\qquad
\forall\alpha,\beta\in\N,\
\eeqs
for some $C_{\alpha,\beta}>0$.

The following theorem holds (cf. \cite{KG}):
\begin{Th}[Sharp-G{\aa}rding]
\label{thA1}
Let $A(x,\xi)\in S^m$ and assume that $\Re A(x,\xi)\geq0$.
There exist then pseudo-differential operators $Q(x,D_x)$ and $R(x,D_x)$
with symbols, respectively, $Q(x,\xi)\!\in S^m$ and $R(x,\xi)\in S^{m-1}$,
such that
\beqs
\nonumber
&&A(x,D_x)=Q(x,D_x)+R(x,D_x)\\
\nonumber
&&\Re\langle Q(x,D_x) u,u\rangle\geq0\qquad\forall u\in H^m\\
\label{A1}
&&R(x,\xi)\sim\psi_1(\xi)D_xA(x,\xi)+
\sum_{\alpha+\beta\geq2}\psi_{\alpha,\beta}(\xi)
\partial_\xi^\alpha D_x^\beta A(x,\xi),
\eeqs
with $\psi_1,\psi_{\alpha,\beta}$ real valued functions,
$\psi_1\in S^{-1}$ and $\psi_{\alpha,\beta}\in S^{(\alpha-\beta)/2}$.
\end{Th}

\begin{Rem}\label{remA}\begin{em}
Terms of the form $\psi_{\alpha,\beta}(\xi)
\partial_\xi^\alpha D_x^\beta A(x,\xi)\in S^{m-(\alpha+\beta)/2}$,
$\alpha+\beta\geq 2$,
of \eqref{A1} can be rearranged so that we have
\beqsn
R(x,\xi)\sim\sum_{\ell\geq 1}R(x,\xi)\big|_{\ord(m-\ell)},
\eeqsn
where
\beqs
\label{Asc}
R(x,\xi)\big|_{\ord(m-\ell)}=\left\{\begin{array}{ll}
\psi_1(\xi)D_xA(x,\xi)+\ds\sum_{2\leq
  \alpha+\beta\leq3}\psi_{\alpha,\beta}(\xi)\partial_\xi^\alpha
D_x^\beta A(x,\xi), & \quad\ell=1,
\\
\ds\sum_{2\ell\leq
  \alpha+\beta\leq2\ell+1}\psi_{\alpha,\beta}(\xi)\partial_\xi^\alpha
D_x^\beta A(x,\xi),
& \quad\ell\geq 2.
\end{array}\right.
\eeqs
\end{em}
\end{Rem}

\begin{Rem}
\label{remA2}
\begin{em}
Theorem \ref{thA1} implies the well-known sharp-G{\aa}rding inequality
\beqs
\label{A2}
\Re\langle A(x,D_x)u,u\rangle\geq-c\|u\|_{(m-1)/2}^2
\eeqs
for some fixed constant $c>0$.
\end{em}
\end{Rem}

A well-known refinement of \eqref{A2} is given by the following theorem (cf. \cite{FP}):
\begin{Th}[Fefferman-Phong inequality]
\label{thA3}
Let $A(x,\xi)\in S^m$ with $A(x,\xi)\geq0$. Then
\beqs
\label{A3}
\Re\langle A(x,D_x)u,u\rangle\geq-c\|u\|_{(m-2)/2}^2\qquad \forall u\in H^m
\eeqs
for some $c>0$ depending only on the constants $C_{\alpha,\beta}$ in \eqref{calphabeta}.
\end{Th}

\begin{Rem}
  \label{remA4}
  \begin{em}
    By \cite{LM} we know that for $m=2$ the constant $c$ in \eqref{A3} depends
    only on $\max_{|\alpha|+|\beta|\leq 7}C_{\alpha,\beta}$, for $C_{\alpha,\beta}$
    as in \eqref{calphabeta}.
  \end{em}
\end{Rem}

\end{appendices}

{\bf Acknowledgment.}
The authors were partially supported by the INdAM-GNAMPA Project 2015
``Equazioni Differenziali a Derivate Parziali di Evoluzione e Stocastiche''.

\end{document}